%% file: Rossi_Schiavoni_Socionovo_interior_singularity_and_branching_geodesics.tex
\documentclass[11pt]{amsart}

\input{preambolo_gg}

\usepackage{hyperref}
\usepackage{tikz}
\usepackage{subfigure}
\usepackage{graphicx}
\usepackage{float}
\usepackage{soul}
\usepackage{xcolor}

\usepackage{bm}

\mathtoolsset{showonlyrefs}

\renewcommand{\k}{\kappa}

\newcommand{\Tan}{\mathrm{Tan}}

\newcommand{\sr}{sub-Riemannian }
\newcommand{\zz}{{\alpha}}
\newcommand{\gse}{{\gamma_{s,\e}}}

\newcommand{\ceq}{\coloneqq}

\usepackage[margin=2.5cm]{geometry}

\newcommand{\mf}{\mathfrak}
\newcommand{\df}{{\mathrm{d}}}

\newcommand{\dsy}{\displaystyle}

\theoremstyle{definition}

\newtheorem{notation}[theorem]{Notation}
\newtheorem*{prob}{Open problem}

\newcommand{\CC}{{\mc C_{s,\e}}}
\newcommand{\CO}{\mathcal C_{s,\e}^{\mr{opt}}}

\newcommand{\AC}{{\rm AC}}
\newcommand{\pr}{\mathrm{pr}}
\newcommand{\lift}{\mathrm{lift}}

\author{T. Rossi}
\address{Università degli Studi dell'Aquila, Via Vetoio, 67100 L'Aquila AQ, IT}
\email{tommaso.rossi1@univaq.it}

\author{A.\,J.\,A. Schiavoni Piazza}
\address{Scuola Internazionale Superiore di Studi Avanzati (SISSA), Via Bonomea, 265, 34136 Trieste TS, IT}
\email{aschiavo@sissa.it}

\author{A. Socionovo}
\address{Unité de Mathématiques Appliquées, ENSTA, Institut Polytechnique de Paris, 91120 Palaiseau, FR}
\email{alessandro.socionovo@ensta.fr}

\begin{document}

\title[Interior singularity and branching of geodesics]{Interior singularity and branching of geodesics in real-analytic sub-Riemannian manifolds}

\begin{abstract}
    We study the regularity and branching of strictly abnormal minimizing geodesics in sub-Riemannian geometry. We construct examples of real-analytic sub-Riemannian manifolds admitting minimizing geodesics that lose regularity at an interior point of their domain and exhibit branching, thereby resolving longstanding open questions. Moreover, using a lifting procedure, we provide the existence of non-smooth and branching minimizing geodesics also in Carnot groups.
\end{abstract}

\maketitle

\setcounter{tocdepth}{1}
\tableofcontents

\section{Introduction}
A sub-Riemannian manifold is a triplet $(M,\De,g)$, where $M$ is a smooth and connected manifold, $\De\subset TM$ is a smooth bracket-generating distribution (i.e., a distribution satisfying the H\"ormander condition), and $g$ is a smooth Riemannian metric on $\De$. A {\em horizontal} (or, {\em admissible}) curve is an absolutely continuous trajectory in $M$ which is tangent to $\De$ almost everywhere. When $\De=TM$, the pair $(M,g)$ is a Riemannian manifold and every absolutely continuous curve is admissible. The length of a horizontal curve $\eta:[0,T]\to M$ is defined as
\begin{equation}
    \label{eq:LSR}
    L_{\rm SR}(\eta)\ceq \int_0^T \sqrt{g(\dot\eta(t),\dot\eta(t))}\,dt.
\end{equation}
The Chow-Rashewskii Theorem ensures that $M$ is horizontally path-connected, and thus the sub-Riemannian distance between two points $p,q\in M$, given by
\begin{equation}
    d_{\rm SR}(p,q)\ceq \inf\{L_{SR}(\eta) \mid \eta:[0,T]\to M \text{ horizontal, } \eta(0)=p, \ \eta(T)=q\},
\end{equation}
is well-defined, continuous, and compatible with the manifold topology. Given $p,q
\in M$, a horizontal curve $\eta:[0,T]\to M$ joining them is a \emph{length-minimizer} if
\begin{equation}
    d_{\rm SR}(p,q)=d_{\rm SR}(\eta(0),\eta(T))=L_{\rm SR}(\eta).
\end{equation}
If, in addition, $\eta$ is parametrized by arc-length, it is referred to as a \emph{minimizing geodesic}. An important class of \sr manifolds is given by Carnot groups, i.e., nilpotent Lie groups with stratified Lie algebra, which naturally arise as the metric tangents of \sr manifolds. 

The study of length-minimizing curves is one of the most important and difficult topics in sub-Riemannian geometry. The Pontryagin Maximum Principle (cf. \cite[Ch.\ 4]{ABB20}) determines first-order necessary conditions for a horizontal curve to be a minimizing geodesic. Every horizontal curve satisfying those conditions is called an {\em extremal curve}, and thus, every minimizing geodesic is an extremal curve. Extremal curves are divided into two non-disjoint classes, called {\em normal} and {\em abnormal} (or {\em singular}) curves. Normal curves are solutions of a geodesic equation, and thus they are smooth and locally minimizing. These are the only types of extremal curves appearing in Riemannian geometry. Instead, abnormal curves are a priori no more regular than absolutely continuous, and their length-minimality properties are unknown. Thus, the difficulty in studying of minimizing geodesics lies in the possible presence of {\em strictly abnormal} extremals, i.e., those curves that are abnormal but not normal. We refer the reader to~\cite{ABB20,J14,libro_Enrico,Mon02,Rif14} for an exhaustive introduction to sub-Riemannian geometry and to the problems related to minimizing geodesics.

In this paper, we deal with two central problems concerning strictly abnormal minimizing geodes-ics: regularity and branching.

\subsection{The regularity problem}

The problem of the regularity of sub-Riemannian geodesics originated in the 1980s with the works of Strichartz in~\cite{Str86, Str89}, where he incorrectly claimed that every sub-Riemannian minimizing geodesic was smooth. Later, in 1994, the first example of a strictly abnormal curve that is length-minimizing was discovered by Montgomery~\cite{Mon94}. After that, many other examples have been found, see~\cite{AS96, LS95}, but all these curves are $C^\oo$-smooth, and they remained the only known instances of strictly abnormal minimizing geodesics for almost 30 years. Beyond these examples, a few general results on strictly abnormal geodesics have been obtained in recent years. On the one hand, building upon \cite{LM08}, it was proved in \cite{HL16} that abnormal minimizing geodesics cannot have corners, and later in \cite{HL23, MPV18}, that they must have a tangent line. In some special cases, including real-analytic three-dimensional or metabelian (cf.\ \cite[Def.\ 3.2]{LPS24}) sub-Riemannian manifolds, minimizing geodesics are of class $C^1$, see \cite{BCJPS20, BFPR22, LPS24}, but in general, even in Carnot groups, the $C^1$ regularity remains an open question. On the other hand, necessary conditions for the minimality of abnormal curves can be derived by the differential analysis of the end-point map (see \cite[Ch.\ 8]{ABB20} for a detailed introduction to the topic). This theory is well understood up to the second order, see \cite{AgrSac, Goh66}, and it has been partially extended to the $n$-th order in \cite{BMS24}. A recent and significant advance in the theory has been made in \cite{notall}, where the authors find sub-Riemannian structures exhibiting minimizing geodesics which are not $C^\oo$ at a boundary point. In particular, the lowest regularity achieved with these examples is that of a minimizing geodesic of class $C^2\setminus C^3$. Nevertheless, a main question is still open: can a sub-Riemannian minimizing geodesic lose regularity at an interior point of its domain?

\subsection{The branching problem}

In the study of metric spaces, detecting whether minimizing geodesics can branch has become increasingly relevant, particularly in connection with the Monge problem of optimal transport (see e.g.\ \cite{MR3425266,MR2984123,MR3027581}). A minimizing geodesic $\eta:[0,T]\to M$ in a metric space $(M,d)$ is \emph{branching} at time $t\in (0,T)$ if there exists a minimizing geodesic $\eta':[0,T]\to M$ such that 
\begin{equation}
    \eta|_{[0,t]} \equiv \eta'|_{[0,t]}\qquad \text{and} \qquad \eta|_{[0,\tau]} \neq \eta'|_{[0,\tau]} \quad \text{for all } \tau\in (t,T].
\end{equation}
Typical examples of branching metric spaces (i.e., metric spaces admitting branching geodesics) are graphs, locally finite CW-complexes, or (sub-)Finsler manifolds equipped with non-strictly convex norms (see \cite{magnabosco2023failure,kensiu,branching-polyhedral} for examples of branching in sub-Finsler geometry). On the contrary, non-branching spaces include Riemannian manifolds, uniformly convex Banach spaces and Alexandrov spaces with curvature bounded from below. Only in recent years have the first examples of branching sub-Riemannian manifolds been discovered in \cite{MR4153911}. In the aforementioned contribution, the authors study branching of normal geodesics, showing that this is equivalent to a jump in the rank of the differential of the end-point map along the curve. Exploiting this characterization, they are able to produce examples of families of strictly normal minimizing geodesics by gluing together the Heisenberg and the Martinet structures. Note that this gluing is not real-analytic; indeed, normal geodesics cannot branch in a real-analytic sub-Riemannian manifold. The work \cite{MR4153911} left a few questions open: firstly, does an example of a strictly abnormal branching minimizing geodesic exist? Secondly, can minimizing geodesics branch in a real-analytic sub-Riemannian manifold?

\subsection{The main results} In this paper, we give an affirmative answer to the regularity and branching questions raised above. On $\R^3$, with coordinates $(x_1,x_2,x_3)$, let $X_1$ and $X_2$ be the vector fields defined by
\begin{equation}
	\label{eq:horvf}
	X_1(x)\ceq \partial_{x_1}, \qquad 
	X_2(x)\ceq \partial_{x_2} + P(x)^2\partial_{x_3},
\end{equation}
where, for $b\in\N$, $b\geq5$ and odd, $P$ is the polynomial defined by
\begin{equation}
	\label{eq:P}
	P(x)\ceq x_1^2-x_2^b,\qquad\forall\, x=(x_1,x_2,x_3)\in\R.
\end{equation}
We denote by $\mc M\ceq (\R^3,\De,g)$ the sub-Riemannian manifold where $\De=\vspan\{X_1,X_2\}$ and $g$ is the metric obtained by declaring $\{X_1,X_2\}$ to be an orthonormal frame. Let us also consider the curves $\g,\bar\g:\R\to\R^3$ defined by
\begin{equation}
	\label{eq:gammaintro} 
	\g(t)\ceq 
	\begin{cases}
	 	\Big(0,t,\frac{t^{2b+1}}{2b+1}\Big), \quad &t<0,
        \\
        (t^q,t,0), \quad &t\geq0,
	\end{cases}
    \quad 
    \text{and}
    \quad
    \bar\g(t)\ceq 
	\begin{cases}
	 	\Big(0,t,\frac{t^{2b+1}}{2b+1}\Big), \quad &t<0,
        \\
        (-t^q,t,0), \quad &t\geq0,
	\end{cases}
    \qquad
    \text{where } q\ceq \frac b2.
\end{equation}
For $s,\e>0$, we denote by $\g_{s,\e}$ and $\bar\g_{s,\e}$ the restrictions of $\g$ and $\bar\g$ to the interval $[-s,\e]$, respectively. When $s=0$, we omit the subscript $s=0$ and write $\g_\e$ and $\bar\g_\e$. In \cite{notall}, it is proved that $\g_\e,\bar\g_\e$ are length-minimizing for sufficiently small $\e>0$. Here, we significantly improve that result.

\begin{theorem}
	\label{thm:main}
	There exist $\e,s>0$ such that the curves $\g_{s,\e},\bar\g_{s,\e}$ are length-minimizing in $\mc M$.
\end{theorem}

Both curves $\gse$ and $\bar\g_{s,\e}$ are of class $C^{\lfloor q \rfloor}\setminus C^{\lfloor q+1 \rfloor}$ at 0, which is an interior point of their domain. The examples with the lowest regularity are obtained for $b=5$, for which $\gse,\bar\g_{s,\e}$ are length-minimizers of class $C^2\setminus C^3$. This regularity is the lowest obtainable through this construction, as shown in \cite{S25}. Since $\bar\g(t)=\g(t)$ for all $t\leq0$, the curves $\gse$ and $\bar \g_{s,\e}$ provide the first example of strictly abnormal branching length-minimizers in a real-analytic (even polynomial) sub-Riemannian manifold. This constitutes our second main result.

\begin{theorem}
\label{thm:john_cena}
    There exist real-analytic sub-Riemannian manifolds admitting branching strictly abnormal minimizing geodesics.
\end{theorem}

We observe that our examples are distributions of Martinet-type in $\mathbb{R}^3$, with the Martinet region that itself branches. In this framework, it is easy to see that branching strictly abnormal curves do exist, while it is far more demanding to prove that they are length-minimizing. Furthermore, our example of branching length-minimizers is ``discrete'', as opposed to \cite{MR4153911} where the authors find a one-parameter family of strictly normal branching geodesics.

Our next main result shows that the minimizing geodesics obtained in Theorem \ref{thm:main} can be lifted to minimizing geodesics in a Carnot group preserving both the regularity and the branching. This gives the first example of its kind.  

\begin{theorem}
    \label{thm:lift}
    There exist Carnot groups admitting minimizing geodesics that are non-smooth at an interior point of their domain and that exhibit branching.
\end{theorem}

Carnot groups are a special class of \sr manifolds, in the sense they are equiregular, self-similar and isometrically homogeneous. This theorem shows that neither equiregularity nor the group structure is enough to rule out non-smooth or branching geodesics.

\subsection{Further directions}
The sub-Riemannian manifold $\mc M$ is real-analytic, of dimension three and metabelian, cf.\ \cite[Def.\ 3.2]{LPS24}. As previously mentioned, minimizing geodesics in three-dimensional or metabelian real-analytic sub-Riemannian manifolds must be of class $C^1$, cf.\ \cite{BFPR22,LPS24}. Thus, it remains an open question whether this result is sharp. More precisely, we have the following. 

\begin{prob}
    Find a minimizing geodesic in a three-dimensional or metabelian real-analytic sub-Riemannian manifold of class $C^1$ but not $C^2$, with a singularity at an interior point of its domain. 
\end{prob}

\subsection{Structure of the paper}
A large part of the paper is devoted to the proof of Theorem \ref{thm:main}. Note that, in light of Proposition \ref{prop:branching} below, it is enough to show that $\gse$ is a length-minimizer, for small $s,\e>0$. The proof of this fact is by contradiction, assuming that there exists a shorter (normal) curve joining the endpoints of $\gse$, and it is divided into several steps, identified by the sections of the paper. 

In Section \ref{sec:prelim}, we reduce the geodesic problem in $\mc M$ to the problem of minimizing the Euclidean length $L(\cdot)$ among plane curves $\omega\in \AC([0,T],\R^2)$, parametrized by arc-length and subject to the constraints $\omega(0) = (0,-s)$, $ \omega(T) = 
(\e^q,\e)$, for some $s,\e>0$, and 
\begin{equation}\label{COS}
	\int_0^{T} \dot\omega_2(t) P(\omega(t))^2  dt=-\frac{s^{2b+1}}{2b+1}.
\end{equation}
By the Pontryagin Maximum Principle, the set of (plane) optimal competitors, denoted by $\CO$, contains only (projections of) normal curves and, thus, every $\w\in\CO$ is real-analytic. Moreover, since $\w\in\CO$ is parametrized by arc-length, there is a real-analytic function $\theta:[0,T]\to \R$ such that $\dot\w_1(t)=\cos\0(t)$ and $\dot\w_2(t)=\sin\0(t)$, and that satisfies
\begin{equation}
	\label{eq:normalpolar-intro}
	\dot\0(t)=\la_\w Q(\w(t)),
\end{equation}
where $\la_\w\in\R$ is a multiplier, and $Q(x)\ceq \partial_{x_1}(P(x)^2)=4x_1P(x)$. Note that the Martinet region $\mathscr S$ is the zero-locus of $Q$, i.e.,
\begin{equation}
    \mathscr S=\{Q=0\}=\{P=0\}\cup \{x_1=0\}.
\end{equation}

Next, in Section \ref{sec:liu_suss}, we crucially improve the Liu--Sussmann-type estimate of \cite[Prop.\ 2.2(ii)]{notall}. This result finds its origin in \cite{LS95} and provides an estimate on the distance of an optimal competitor from $\mathscr S$ in terms of $\e>0$. More precisely, define 
\begin{equation}
    \wt\beta_\w \ceq \max_{[0,L(\w)]}|\wt P(\w(t))|,
\end{equation}
where $\wt P:\R^2\to\R^2$ is the function
    \begin{equation}
    \label{eq:Ptilde}
        \wt P(x_1,x_2)\ceq 
        \begin{cases}
            P(x), \quad &x_2\geq0,
            \\
            x_1^2, \quad &x_2<0.
        \end{cases}
    \end{equation}
Then, $\wt\beta_\w=0$ if and only if $\w$ is the projection of a curve in $\mathscr S$ and, in Corollary \ref{cor:newLS}, we prove that $\wt \beta_\w=o(\e^{3q-1})$ as $\e\to 0$. Compared to \cite{LS95,notall}, where the latter estimate is obtained using that $P=0$ along the \emph{whole} abnormal curve, our strategy relies instead on the minimality of $\gamma_\e$. 

Using the improved Liu--Sussmann-type estimate, we provide a precise description of optimal competitors, with the final aim of comparing their length with the length of $\gse$. Even though this study is similar to the one in \cite{notall}, there are some non-trivial difficulties to be accounted for. Indeed, the ODE \eqref{eq:normalpolar-intro} for an optimal competitor naturally involves $P$. Instead, since $\gse$, with $s>0$, is not supported in $\{P=0\}$, the right quantity to consider for length estimates is $\wt \beta_\w$, which is defined using $\wt P$. This requires an ad hoc adaptation of most of the arguments in \cite{notall}, together with new crucial ideas to complete the qualitative analysis of $\w$. The goal of this analysis, carried out in Sections \ref{sec:level_sets} to \ref{sec:length_estimates}, is to show that:
\begin{enumerate}[i)]
	\item the multiplier $\lambda_\w$ in \eqref{eq:normalpolar-intro} is negative;
	
	\item the competitor is supported in the positive region of $\wt P$, i.e., $\wt P(\w(t))>0$ for all $t\in (0,L(\w))$;
	
	\item the competitor has one single loop, which is quantitatively far away from the starting point;

    \item the function $t\mapsto \w_1(t)$ is strictly increasing up to the $x_1$-axis;
    
	\item the maximum of the function $t\mapsto \wt P(\w(t))$ is achieved inside the loop.
\end{enumerate}

More in detail, in Section \ref{sec:level_sets}, we deduce length estimates of the level sets of $\wt P$. Here, a new argument tailored to $\wt P$ is needed. In Section \ref{sec:geom_obs}, we report some geometric obstructions to optimality taken from \cite{notall}. These obstructions are local in nature and do not depend on the starting point. In Sections \ref{sec:looooooooops} and \ref{sec:looooooop}, we study loops of optimal competitors. The main challenge here is to localize the loops to the region $\{x_1\geq C\e^q,x_2\geq C\e\}\cap \{P>0\}$, cf.\ Proposition \ref{prop:loops>eps/2} and Lemma \ref{lem:svolta}. This localization procedure is crucial to translate estimates for $P$ into estimates for $\wt P$, and to show that there is a unique loop. Additionally, in Corollary \ref{cor:unique_intersection_x1axis}, we show that the first component of any optimal competitor is strictly increasing up to the intersection with the positive $x_1$-axis, which is unique and transversal. In Section \ref{sec:length_estimates}, we exploit Corollary \ref{cor:unique_intersection_x1axis} to relate the maxima of $P$ and of $\wt P$ along a competitor, see Lemma \ref{lem:loc_max_P}. This allows us to prove that $\wt P\circ\w>0$ and that the maximum of $\wt P\circ\w$ is achieved inside the loop. 

We conclude the proof of Theorem \ref{thm:main} in Section \ref{sec:proof}. The final contradiction argument is based on length estimates for the unique loop of an optimal competitor, which can be obtained similarly to those in \cite{notall}, once all the intermediate ingredients have been established for $\wt \beta_\w$ and $\wt P$.

Finally, the last section of the paper, Section \ref{sec:lift}, is devoted to the proof of Theorem \ref{thm:lift}, where we explicitly construct a Carnot group $G$, together with a submetry from $G$ to $\mc M$. This submetry allows then to lift the curves in $\mc M$ to curves in $G$ while preserving minimality.

\subsection*{Acknowledgments} The authors are grateful to E. Le Donne, N. Paddeu, and D. Vittone for fruitful discussions on the lift of geodesics in Carnot groups. The second author was partially supported by the European Research Council (ERC) under the European Union’s Horizon 2020 research and innovation programme (grant agreement GEOSUB, No.~945655). The first and second authors are members of GNAMPA, INdAM.

\section{Preliminaries and reduction to \texorpdfstring{$\mathbb{R}^2$}{R2}}
\label{sec:prelim}

We reduce the study of horizontal curves in $\mc M$ and their length to the study of curves in $\R^2$. An absolutely continuous curve $\eta:[0,T]\to\R^3$ is horizontal in $\mc M$ if and only if
\begin{equation}
    \label{eq:horizontal_curve}
    \dot\eta(t)=\dot\eta_1X_1(\eta(t))+\dot\eta_2X_2(\eta(t)), \quad \text{for a.e. } t\in[0,T].
\end{equation}
Consider the projection onto the first two coordinates $\pi:\R^3\to\mathbb{R}^2$, i.e., $\pi(x_1,x_2,x_3)\ceq (x_1,x_2)$. For any $T>0$, define a projection map $\pr:\AC([0,T],\R^3)\to\AC([0,T],\mathbb{R}^2)$, as
\begin{equation}
\label{eq:projection}
    \pr(\eta)\ceq \pi\circ\eta,\qquad\forall\eta\in\AC([0,T],\R^3).
\end{equation}
Denote by $L(\cdot)$ the Euclidean length of a plane curve. Since $\{X_1,X_2\}$ is an orthonormal frame for $\Delta$, with respect to $g$, from \eqref{eq:LSR} and \eqref{eq:horizontal_curve} it follows that 
\begin{equation}
    \label{eq:LSR=L}
    L_{\rm SR}(\eta)=\int_0^T \sqrt{\dot\eta_1(t)^2+\dot\eta_2(t)^2}dt=L(\pr(\eta)),
\end{equation}
for every horizontal path $\eta$. In particular, for all $s>0$ and as $\e\to 0$, the length of the curve $\gse=\gamma|_{[-s,\e]}$, see \eqref{eq:gammaintro}, is
\begin{align}
    \label{eq:length_abnormal_curve}
    L_{\rm SR}(\g_{s,\e})&=s+\int_{0}^{\e} \sqrt{1 + q^ 2 t^{2q-2}  } dt=s+\e+\frac{q^2}{2(2q-1)}\e^{2q-1}+o(\e^{2q-1}).
\end{align}
If we restrict to horizontal curves, the projection map admits a right inverse.
Indeed, given a plane curve $\w=(\w_1,\w_2)\in \AC([0,T],\R^2)$, set 
\begin{equation}
	\label{eq:dotomega3}
	\omega_3(\w,t) \ceq  \int_0^t\dot \w_2(\tau) P(\w(\tau))^2d\tau,\quad t\in [0,T].
\end{equation}
Then, the function $\lift:\AC([0,T],\mathbb{R}^2)\to\AC([0,T],\R^3)$, defined as
\begin{equation*}
    \w\mapsto\lift(\omega)\ceq (\omega_1,\omega_2,\omega_3(\w,\cdot)),
\end{equation*}
takes values in the set of horizontal curves and trivially satisfies $\pr\circ\lift=\rm{Id}$. Note that curves lifted in this way will always satisfy $\omega_3(0)=0$. Adding a constant to the third component gives a different right inverse for the projection function.

\begin{proposition}
    \label{prop:branching}
    Let $\phi:\R^3\to\R^3$ given by $\phi(x)=(-x_1,x_2,x_3)$. A horizontal curve $\eta:[0,1]\to \R^3$ is length-minimizing in $\mc M$ if and only if the curve $\bar\eta\ceq \phi\circ\eta_3$ is. 
\end{proposition}

\begin{proof}
    We first prove that $\phi$ preserves horizontal curves in $\mc M$, as well as the variation on the third coordinate. Indeed, from \eqref{eq:horizontal_curve}, we have
    \begin{equation}
        \dot\eta_3=\dot\eta_2P(\eta_1,\eta_2)^2 =\dot\eta_2 P(-\eta_1,\eta_2)^2=\dot{\bar\eta}_2P(\bar\eta_1,\bar\eta_2)^2=\dot{\bar\eta}_3.
    \end{equation}
    Then, the fact that $\phi$ is an isometry of $\mc M$ follows from \eqref{eq:LSR=L}.
\end{proof}

As a consequence, we may focus on proving Theorem \ref{thm:main} for $\gse$. Moreover, according to the discussion above, the problem of finding length-minimizing curves between two points in $\mc M$ is equivalent to the problem of minimizing the Euclidean length of curves between their projections in $\mathbb{R}^2$, under the additional constraint given by the variation of the third coordinate.

\begin{definition}
	\label{def:competitor} Given $s,\e>0$, we say that  $\w\in \AC([0,T],\R^2)$ is a {\em competitor} for $\gamma_{s,\e}$, and we write $\w\in\CC$, if the following conditions hold:
\begin{enumerate}[(i)]
\item $\w(0)=\pi(\g(-s))$, $\w(T)=\pi(\g(\e))$;
\item\label{item:def_competitor_ii} $\omega_3(\omega,0)=-\frac{s^{2b+1}}{2b+1}$;
\item $T= L(\w)$ and $\w$ is parametrized by arc-length;
\item $\w$ is different from the arc-length reparametrization of $\pr(\gse)$.
\end{enumerate}
If, additionally, $L(\omega)=d_{\rm SR}(\gamma(-s),\gamma(\e))$, we say that $\omega$ is an \emph{optimal competitor} for $\gamma_{s,\e}$ and we write $\w\in\CO$.  
\end{definition}

We recall that the Pontryagin Maximum Principle (cf. \cite[Ch.\ 4]{ABB20}) implies that the lift of any $\w\in\CO$ must be a normal or an abnormal curve, joining the points $\g(-s)$ and $\g(\e)$. 

\begin{remark}
    \label{rem:study_only_normal}
    Denote by $\mc N_{s,\e}$ the set of normal curves, parametrized by arc-length, and joining $\gamma(-s)$ and $\gamma(\e)$. Denote by $\mc N_{s,\e}^{\mathrm{opt}}$ the set of curves in $\mc N_{s,\e}$ which are length-minimizing. Since $\gamma_{s,\e}$ is the only abnormal trajectory joining its endpoints (up to reparametrization),  condition (iv) above implies that $\pr(\mc N_{s,\e}^{\rm{opt}})=\mathcal C_{s,\e}^{\rm{opt}}$.
\end{remark}

For any competitor $\w\in \CC$, we denote $I_\w\ceq [0,L(\w)]$. For any normal competitor $\w\in \pr(\mc N_{s,\e})$, we define the {\em angle map} $\0=\0_\w:I_\omega\to\R$ as the unique $C^\oo$ function  such that $\0(0)\in[-\pi,\pi)$ and
    \begin{equation}
	   \label{eq:teta}
	   \dot\w_1(t)=\cos(\0(t)), \quad \dot\w_2(t)=\sin(\0(t)), \quad \text{for all } t\in I_\omega.
    \end{equation}

\begin{lemma}	
	\label{lem_polarnormal}
	Fix $\e,s>0$. For all $\w\in\CO$ there exists  $\la=\lambda_\w\in\R$ such that 
	\begin{equation}
		\label{eq:normalpolar}
		\dot\0(t)=\la Q(\w(t)), \quad  t\in I_\omega.
	\end{equation}
	The quantity $\la Q(\w(t))$ is the curvature of $\w$ at the point $\w(t)$. 
\end{lemma}

\begin{proof}
    On the one hand, by \eqref{eq:teta}, $\dot\theta$ is the signed curvature of $\w$. On the other hand, since $\pr(\mc N_{s,\e}^{\rm{opt}})=\mathcal C_{s,\e}^{\rm{opt}}$, see Remark \ref{rem:study_only_normal}, we can apply the Pontryagin Maximum Principle to any lift of $\w$. Hence, recalling that the \sr Hamiltonian is
    \begin{equation}
        H(\lambda)=\frac12(h_1^2(\lambda)+h_2^2(\lambda))=\frac12\left(p_1^2+\left(p_2+P(x_1,x_2)^2p_3\right)^2\right),\qquad\forall\, \lambda=(p_1,p_2,p_3)\in T^*M, 
    \end{equation}
    where $h_i(\lambda)=\langle \lambda,X_i\rangle$, we have that any lift $(\omega_1,\omega_2,\omega_3)$ of $\w$  must satisfy  
    \begin{equation}
        \left\{\begin{array}{rlcrl}
             \dot p_1& \!\!= -\partial_1(P^2)(\omega(t)) h_2(t) p_3(t), & \qquad &\dot \w_1& \!\!= h_1(t),  \\
             \dot p_2& \!\!= -\partial_2(P^2)(\omega(t)) h_2(t) p_3(t), & \qquad &\dot \w_2& \!\!= h_2(t),\\
             \dot p_3& \!\!=0, & \qquad &\dot \w_3& \!\!= P(\omega(t))^2 \dot\w_2(t).
        \end{array}\right.
    \end{equation}
    Therefore, we easily see that, for every $t\in I_\w$, 
    \begin{equation}
    \begin{split}
        \left(\ddot\w_1(t),\ddot\w_2(t)\right)= \left(-Q(\w(t))h_2(t)p_3(t),Q(\w(t))h_1(t)p_3(t)\right)= p_3(t)Q(\w(t))\left(-\dot \w_2(t),\dot\w_1(t)\right).
    \end{split}
    \end{equation}
    As $\w$ is parametrized by arc-length, this means that its signed curvature at time $t$ is given by $Q(\w(t))p_3(t)$. Since $p_3(t)\equiv p_3(0)$, we conclude the proof. 
\end{proof}

For the next lemma, recall that normal curves in $\mc M$ are real-analytic, see for instance \cite[Section 4.3.1]{ABB20} or \cite[Section 2.2]{Rif14}. Hence, by Remark \ref{rem:study_only_normal}, any $\omega\in\mathcal{C}_{s,\e}^{\rm{opt}}$ is real-analytic.   
\begin{lemma}\label{lem:w1>0} 
    For every $s,\e>0$, $\omega\in \CO$, and $t \in(0,L(\w)]$, it holds $\w_1(t)>0$. In addition, $\0(0)\in(-\frac \pi 2,\frac \pi 2)$ and there exists $r=r(\w)>0$ such that $\widetilde{P}(\w(t))>0$ for $t\in(0,r)$.
\end{lemma}

\begin{proof}
	We first prove that $\w_1\geq0$ on $I_\omega$. Assume by contradiction that there exists $t^*\in I_\omega$ such that $\w_1(t^*)<0$. Set, for the sake of notation, $\tau\ceq L(\w)$ and define
	\[
	t_0\ceq \max\{ t\in [0,t^*] \mid \omega_1(t)=0\} \quad\textrm{and}\quad
	t_1\ceq \min\{ t\in [t^*,\tau] \mid \omega_1(t)=0\},
	\]
	and let $\eta\in \AC([0,\tau],\R^2)$ be the curve 
	\begin{equation}
		\eta (t)=
		\begin{cases}
			\w(t), \quad &t\in[0,\tau]\setminus[t_0,t_1],\\
			(-\w_1(t),\w_2(t)), \quad &t\in[t_0,t_1].
		\end{cases}
	\end{equation}
	By construction, $L(\eta)=L(\w)$, $\eta(0)=\w(0)$, and $\eta(\tau)=\w(\tau)$. Moreover, $P\circ\eta=P\circ\w$ and thus 
	\begin{equation}
		\omega_3(\eta,\tau)=	\int_0^\tau \dot\eta_2(t)P(\eta (t))^2dt=	\int_0^\tau \dot\w_2(t)P(\w(t))^2dt =\omega_3(\w,\tau).
    \end{equation}
	Therefore, $\lift(\eta)$ is an optimal curve which is not contained in the Martinet surface and is not a real-analytic extremal, which is a contradiction. This proves that $\w_1\geq0$. In addition, since $\w_1(0)=0$ and $\w_1\geq 0$, then $\dot\w_1(0)=\cos(\theta(0))\geq 0$, showing that $\0(0)\in\left[-\frac\pi2,\frac\pi2\right]$.
	
	We now prove that $\w_1|_{(0,L(\w))}>0$ (note that $\w_1(L(\w))>0$ by definition). By contradiction, assume that there exists $t_0\in (0,L(\w))$ such that $\w_1(t_0)=0$. Then, since $\w_1$ has minimum in the interior point $t_0$, $\dot\w_1(t_0)=\cos(\theta(t_0))=0$. Combining this with the PMP, we see that there exists a curve $p=(p_1,p_2):I_\w\to\R^2$ such that $(p,\w)$ is the unique solution of the problem 
    \begin{equation}
    \label{eq:proj_ham_sys}
        \left\{\begin{array}{rlcrl}
             \dot p_1& \!\!= -\partial_1(P^2)(\omega(t)) h_2(t) p_3, & \qquad &\dot \w_1& \!\!= h_1(t),  \\
             \dot p_2& \!\!= -\partial_2(P^2)(\omega(t)) h_2(t) p_3, & \qquad &\dot \w_2& \!\!= h_2(t),
        \end{array}\right.
    \end{equation}
    with initial conditions at time $t_0$ given by $(0,p_2^{t_0};0,\w_2^{t_0})$. One can check that the (projection of the) unique solution to \eqref{eq:proj_ham_sys} is the straight line $[t_0,L(\w)]\ni t\mapsto (0,\w_2^{t_0}+\sin(\theta(t_0))(t-t_0))$. This gives a contradiction, as $\w_1(L(\w))>0$. An analogous argument shows that at time $t=0$, one cannot have $\cos(\theta(0))=0$. In conclusion, we must have $\w_1(t)>0$, for every $t\in (0,L(\w)]$, and $\0(0)\in(-\frac\pi2,\frac\pi2)$. 

    Finally, since the initial point of $\w\in\mathcal{C}_{s,\e}^{\rm{opt}}$ is $\w(0)=\pi(\gamma(-s))=(0,-s)$, with $\0(0)\in(-\frac\pi2,\frac\pi2)$, it follows that there exists $r>0$ such that $\omega_1(t)>0$ and $\omega_2(t)<0$, for $t\in(0,r)$. As a consequence,  $\widetilde P(\w(t))>0$, for $t\in(0,r)$. 
\end{proof}

\section{A new Liu--Sussmann Lemma}
\label{sec:liu_suss}

A key ingredient in the proof of optimality of the abnormal curve in \cite{notall,LS95} is an estimate on the maximum value of $P$ along an optimal competitor. We study the analogous quantities in our case.

\begin{definition}
    Let $\e,s>0$ and let $\w\in\CO$. We define the two quantities
\begin{equation}
\begin{split}
    &\beta=\beta_\w\ceq \max_{t\in I_\w}|P(\w(t))|,\\
    &\wt\beta=\wt\beta_\w\ceq \max_{t\in I_\w}|\wt P(\w(t))|,
\end{split}
\end{equation}
where $P$ and $\wt P$ are defined in \eqref{eq:P} and \eqref{eq:Ptilde}, respectively.
\end{definition}

We recover the asymptotic estimate on $\beta$ and $\wt\beta$, which generalizes \cite[Lem.\ 3.5(v)]{notall}. 

\begin{lemma}
    \label{lem:quelloforte}
    There exists $\e_\star>0$, such that the following holds. For every $0<\e\leq\e_\star$ and $\delta>0$ there exists $s(\delta,\e)>0$ such that, for every $0<s<s(\delta,\e)$ and every $\omega\in\CO$, it holds
    \begin{equation}\label{Haus_conv_curves}
        {\rm dist}_{\rm Haus}\left({\rm spt}(\omega),{\rm spt}(\pr(\gamma_{\e}))\right)\leq\delta,
    \end{equation}
    where $ {\rm dist}_{\rm Haus}(A,B)$ is the Hausdorff distance between $A,B\subset \R^n$, with respect to the Euclidean distance.    
\end{lemma}

\begin{proof}
    There exists $R>0$ such that the sub-Riemannian ball $\overline{B}_r(\g(0))$ is compact for every $r\leq R$. Fix $r<R$, and note that any horizontal curve $\eta$ that joins a point $p\in \overline{B}_r(\g(0))$ with a point $q\notin\overline{B}_R(\g(0))$ satisfies $L_{SR}(\eta)\geq R-r$. 
    
    Fix $\e_\star,s_\star>0$ such that $L_{SR}(\g_{s_\star,\e_\star})<R-r$, $\spt(\g_{s_\star,\e_\star})\subset\overline{B}_r(\g(0))$, and $\g_{\e_\star}$ is the unique geodesic between its endpoints, up to reparametrization (\cite[Thm.\ 1.1]{notall} ensures that such $\e_\star$ exists). Now, pick $\e<\e_\star$ and consider the compact sets $K\ceq \{\gamma(\e)\}$ and
    \begin{equation}
        K_s\ceq \left\{\gamma(\tau)\,\vert\, \tau\in[-s,0]\right\},\qquad \text{for } 0\leq s\leq s_\star. 
    \end{equation}
    Define the set $\Gamma_{K_s\to K}\subset\AC([0,1], \R^3)$ of minimizing geodesics, parametrized by constant speed, starting from $K_s$ and ending in $K$. Note that a curve $\eta\in\Gamma_{K_s\to K}$ has speed $d_{SR}(\eta(0),\eta(1))$, therefore the family $\Gamma_{K_s\to K}$ is equi-Lipschitz (with respect to $d_{SR}$), with equi-Lipschitz constant bounded by
    \begin{equation}
    L\ceq \max_{p\in K_s,\, q\in K}d_{SR}(p,q)\leq L_{SR}(\g_{s,\e})\leq L_{SR}(\g_{s_\star,\e_\star})<R-r.
    \end{equation}
    Moreover, any curve of the family is supported in $\{p\in {\mc M}\,|\,d_{SR}(p,K)\leq L\}\subset \overline{B}_R(\g(0))$.
    Since $\left(\overline{B}_R(\g(0)), d_{SR}\right)$ is a compact metric space, by the Ascoli-Arzel\`a Theorem, $\Gamma_{K_s\to K}$ is relatively compact with respect to the uniform topology (induced by $d_{SR}$). In addition, by the lower semi-continuity of the length, the set $\Gamma_{K_s\to K}$ is closed, hence compact. 
    
    Let $\tilde{\gamma}_{\e}$ be the reparametrization of $\gamma_{\e}$ with constant speed equal to $L_{SR}(\gamma_{\e})$. Then, for any $0\leq s\leq s_{\star}$, we have $\tilde{\gamma}_{\e}\in\Gamma_{K_s\to K}$ and there exists a curve $\eta_s\in\Gamma_{K_s\to K}$ such that
    $ \|\eta_s-\tilde{\gamma}_{\e}\|_\infty=\max_{\eta\in\Gamma_{K_s\to K}}\|\eta-\tilde{\gamma}_{\e}\|_\infty$, where the uniform norm is computed with respect to $d_{SR}$. We claim that 
    \begin{equation}
    \label{eq:claim_millemila}
        \|\eta_s-\tilde{\gamma}_{\e}\|_\infty\to0, \quad s\to0.
    \end{equation}
    By contradiction, assume that there exist $\delta_0>0$ and a sequence $s_n\to 0$ such that
    \begin{equation}
    \label{eq:AA}
        \|\eta_{s_n}-\tilde{\gamma}_{\e}\|_\infty\geq\delta_0.
    \end{equation}
    Reasoning similarly as before, the sequence $\{\eta_{s_n}\}$ is uniformly bounded and equi-Lipschitz. By Ascoli-Arzel\`a Theorem, there is a converging sub-sequence. By lower semi-continuity of the distance and since $K_s\to\{\g(0)\}$, as $s\to0$, the limit is a minimizing geodesic parametrized by constant speed and joining $\g(0)$ and $\g(\e)$. Thanks to \cite[Thm. 1.1]{notall}, the limit must be $\tilde\g_\e$, contradicting \eqref{eq:AA}. 
    
    In addition, recall that the set of optimal competitors $\mc C_{s,\e}^{\rm opt}$ contains curves parametrized by arc-length, which are projections onto $\mathbb{R}^2$ of minimizing geodesics joining $\pi(\g(-s))$ and $\pi(\gamma(\e))$. The projection $\pr$, see \eqref{eq:projection}, is continuous, hence uniformly continuous on the compact set $\overline{B}_R(\g(0))$. As a consequence, it preserves uniform convergence of bounded family of curves, thus from \eqref{eq:claim_millemila}, it follows that
    \begin{equation}
        \max_{\eta\in\Gamma_{K_s\to K}}\|\pr(\eta)-\pr(\tilde{\gamma}_{\e})\|_\infty\to 0,\quad\text{as }s\to 0,
    \end{equation}
    where now the uniform norm is computed with respect to the Euclidean distance on $\mathbb{R}^2$.
    Since convergence in the uniform topology implies Hausdorff convergence of supports, we get
    \begin{equation}\label{unif_convergence_geodesics}
        \max_{\eta\in\Gamma_{K_s\to K}}\mathrm{dist}_{\mathrm{Haus}}(\spt(\pr(\eta)),\spt(\gamma_{\e}))\to 0, \quad\text{as } s\to 0,
    \end{equation}
    where we used that $\spt(\pr(\tilde{\gamma}_{\e}))=\spt(\tilde{\gamma}_{\e})=\spt(\g_\e)$. We conclude the proof as \eqref{unif_convergence_geodesics} is equivalent to the thesis.   
\end{proof}

\begin{corollary}
\label{cor:newLS}
Let $\e_\star>0$ be the parameter given by Lemma \ref{lem:quelloforte}. For every $0<\e\leq\e_\star$ and $M>0$, there exists $\bar s=\bar s(\e,M)>0$ such that, for all $0<s<\bar s$ and $\w\in\CO$, we have
    \begin{equation}
        \beta_\w<\e^M \quad \text{and} \quad \wt\beta_\w<\e^M.
    \end{equation}
\end{corollary}

\begin{proof}
    We prove the claim only for $\beta_\w$, as the other estimates can be obtained analogously. Firstly, observe that $P\circ\gamma_{\e}\equiv 0$ and that $P$ is locally Lipschitz. In particular, up to restricting to $s\leq 1$, $P$ is Lipschitz on a bounded set $A\subset \R^2$ that contains all the projections onto $\mathbb{R}^2$ of minimizing geodesics in $\Gamma_{K_1\to K}$.  For every $\w\in\CO$, we have  
    \begin{equation}
    \label{dis_Haus_P}
        |P(\w(t))| \leq {\rm Lip}\big(P|_{\bar A}\big){\rm dist}\left(\w(t),{\rm spt}(\gamma_{\e})\right)\leq {\rm Lip}\big(P|_{\bar A}\big)\mathrm{dist}_{\mathrm{Haus}}(\spt(\w),\spt(\gamma_{\e})).
    \end{equation}
    Once $\e$ is fixed, by Lemma \ref{lem:quelloforte}, the quantity in the right-hand side of \eqref{dis_Haus_P} goes to $0$ with $s$, uniformly with respect to the choice of $\w\in\CO$. Therefore, choosing $s_0$ small enough, the thesis follows.
\end{proof}

\begin{notation}
\label{notation:the}
    From now on, we denote by $\e_\star>0$ the parameter given by Lemma \ref{lem:quelloforte} such that, for all $0<\e<\e_\star$, the curve $\g_\e$ is the unique length-minimizer between $(0,0,0)$ and $(\e^q,\e,0)$. We also fix
    \begin{equation}
        M>3q-1>b. 
    \end{equation}
    Given $\e_0>0$, we set\footnote{The requirement $s<\e^2$ is justified by Lemma \ref{lem:geometric_w2>e2} and simplifies the notation in the sequel.} 
    \begin{equation}
        \label{eq:notation}
        \mc I_{\e_0} \ceq \left\{(\e,s)\in\R_+^2 \mid \e<\min\{\e_\star,\e_0\}, \ s < \min\{\bar s(\e,M),\e^2\} \right\},
    \end{equation}
    where $\bar s(\e,M)>0$ is given by Corollary \ref{cor:newLS}.
    By construction, we have $\mc I_{\e_0}\neq \emptyset$. Moreover, if $(\e,s)\in \mc I_{\e_0}$ and $\CO\neq\emptyset$, then we have
    \begin{enumerate}
        \item\label{alex} the curve $\g_\e$ is the unique length-minimizer between $(0,0,0)$ and $(\e^q,\e,0)$;

        \item\label{sandro} for every $\w\in\CO$, we have $\beta_\w,\widetilde{\beta}_\w<\e^M$ and $L(\w)<L_{\mathrm{SR}}(\g_{s,\e})$;
    \end{enumerate}
\end{notation}

We will see in the sequel that optimal competitors must have exactly one point of self-intersection. For this reason, we need to define loops of optimal competitors.

\begin{definition}
    \label{def:loops}
    Fix $\e,s>0$ and $\w \in \CC$. A {\em loop} $\ell$ of $\omega$ is a restriction of $\w$ to some non-trivial interval $J_\ll=[s_\ell^{-},s_{\ell}^{+}]$ such that $\omega(s_\ell^{-})=\omega(s_\ell^{+})$. In addition, $\ell$ is a simple loop if $\omega|_{[s_\ell^{-},s_\ell^{+})}$ is injective.
\end{definition}

Note that if $\omega\in\mathcal{C}_{s,\e}^{\rm{opt}}$, then $\w$ is real-analytic, and thus it has at most a finite number of self-intersections (and thus of loops).

\section{Length estimates on the level sets of \texorpdfstring{$\protect\widetilde{P}$}{\~P}}
\label{sec:level_sets}

We study the level sets of $\wt P$ to deduce some length estimates. Let $\z>0$, and define 
 \begin{equation}
	\mc P_\z\ceq \{x\in\R^2\mid x_1\geq 0,\,\widetilde{P}(x)=\z\} \quad \textrm{and} \quad D_\z\ceq \{x\in\R^2\mid x_1\geq 0,\, \widetilde{P}(x)\leq \z\}.
\end{equation} 
Set $f_{\z}(t)\ceq \sqrt{t^b+\z}$ for all $t\geq 0$. Then, the curve $\G_{\z}:\R_+\to\R^2$, $\G_\z(t)\ceq (f_\z(t),t)$ parametrizes $\mc P_{\z}\cap\{x\in\R^2\mid x_2\geq0\}$. For $\e,s>0$ and $\z>0$, consider the minimization problem:
\begin{equation}\label{MIN}
\min \big\{L(\nu) \mid \nu\in \AC([0,1], \R^2),\,\nu(0)=\pi(\g(-s)),\, \nu (1) =\pi(\gamma(\e)), \, \mathrm{spt}(\nu)\subset D_\z\big\}, 
\end{equation} 
where  $\mathrm{spt}(\nu)=\nu([0,1])$ is the support of $\nu$. By lower semi-continuity of the length and Ascoli-Arzelà Theorem, a minimizer exists. Its uniqueness follows from its explicit construction. Let $\nu_{s,\e}$ be a minimizer for \eqref{MIN}. For $\z>0$ small enough, $\nu_{s,\e}$ is not the segment joining $\pi(\g(-s))$ and $\pi(\g(\e))$ as the latter is not contained in $D_\z$. Then, $\nu_{s,\e}$ can be determined by the following observations: 

\begin{enumerate}[(i)]
\item In the interior of $D_\z$, $\nu_{s,\e}$ is locally a line segment;

\item If $\nu_{s,\e}(t_1),\nu_{s,\e}(t_2)\in\mathrm{spt}(\Gamma_\z)$, for some $t_1\neq t_2$, then the sub-arc of $\Gamma_\z$ connecting them is in $\mathrm{spt}(\nu_{s,\e})$. This follows from the convexity of the function $f_\z$;

\item The curve $\nu_{s,\e}$ must cross $\spt(\Gamma_{\varrho})$ and each maximal line segment in $\mathrm{spt}(\nu_{s,\e})$ meets $\mc P_\z$ tangentially. This follows by a simple deformation argument that is omitted.
\end{enumerate}
We now construct $\nu_{s,\e}$. Firstly, $\nu_{s,\e}|_{[0,t]}$ is a line segment for sufficiently small $t>0$, hence we define
\begin{equation}
    t_0 = t_0(s,\e,\varrho) \ceq  \max\left\{t>0\mid \nu_{s,\e}|_{[0,t]} \text{ is a line segment}\right\}\in (0,1).
\end{equation}
Observe that, by construction, $\nu_{s,\e}(t_0)\in \mc P_{\varrho}$. Moreover, since $\mc P_\z=V_\varrho\cup\spt(\G_\varrho)$, where $V_\z\ceq \left\{(\sqrt\z,x_2)\,\vert\,x_2\leq0\right\}$, and $\nu_{s,\e}(0)=\g(-s)\not\in V_\z$, the tangency point $\nu_{s,\e}(t_0)$ must belong to $\mathrm{spt}(\Gamma_\z)$. In particular, it holds $\nu_{s,\e}(t_0)=(f_\z(y_0),y_0)$ for some $y_0=y_0(s,\e,\z)>0$ and
	\begin{equation}
		\label{eq:t_0}
			\sqrt{y_0^b+\z}=m_0(y_0+s) \qquad \mr{with} \qquad
			m_0\ceq qy_0^{b-1}(y_0^b+\z)^{-\frac12}.
	\end{equation}
Note that $y_0$ is monotone decreasing in $s$ and $y_0\to 0$ as $s\to-\oo$. When $s=0$, \eqref{eq:t_0} reduces to
\begin{equation} \label{zeta_t_0}
 \z=(q-1) y_0(0,\e,\varrho)^b.
\end{equation}
Similarly, since $\pi(\g(\e))\notin\mc P_\varrho$, $\nu_{s,\e}|_{[t,1]}$ is a line segment as $t\to 1$. Hence, we define
\begin{equation}
    t_1 = t_1(s,\e,\varrho) \ceq  \min\left\{t>0\mid \nu_{s,\e}|_{[t,1]} \text{ is a line segment}\right\}\in [t_0,1).
\end{equation}
It holds $\nu_{s,\e}(t_1)\in\mathrm{spt}(\Gamma_\z)$, yielding $\nu_{s,\e}(t_1)=(f_\z(y_1),y_1)$ for some $y_1=y_1(s,\e,\varrho)>0$ and 
	\begin{equation}
		\label{eq:t1}
  \sqrt{y_1^b+\z}=m_1(y_1-\e) + \e^q \qquad \mr{with} \qquad
			m_1\ceq qy_1^{b-1}(y_1^b+\z)^{-\frac12}.
	\end{equation} 
Finally, for $\z$ small enough, we have $0<t_0<t_1<1$ and we set
\[
g_\z(t) \ceq  \left\{
\begin{array}{ll}
m_0(t+s) & \textrm{for } t\in[-s,y_0],
\\
f_\z(t)  &\textrm{for } t \in[y_0,y_1],
\\
m_1 (t-\e)+ \e^q &\textrm{for } t\in [y_1,\e].
\end{array}
\right.
\]
The curve $\tilde{\nu}_{s,\e}(\cdot,\z):[-s,\e]\to\mathbb{R}^2$ defined by $\tilde{\nu}_{s,\e}(t)\ceq (g_\z(t),t)$ is a reparametrization of $\nu_{s,\e}(\cdot,\z)$, which is then the unique solution to \eqref{MIN}.

\begin{proposition} \label{prop:sublevel_length_estimate} Let $C>0$ be a fixed constant. There exist $K,\e_0 >0$ such that for all $\e<\e_0$, $0<\z\leq C \e^{3q-1}$, and $s\geq0$, we have
	\begin{equation}\label{STIMA}
		L(\nu_{s,\e}(\cdot,\z) )\geq L(\g_{s,\e}) - K\z^{1 - \frac 1b},
	\end{equation}    
	where $\nu_{s,\e}(\cdot,\z)$ is the solution to \eqref{MIN}.
\end{proposition}

The proof of Proposition \ref{prop:sublevel_length_estimate} requires a preliminary lemma. Let us introduce the parameters
\begin{equation}
\label{alpha_xi}
\alpha=\alpha(\e,\varrho) \ceq \frac{\z}{\e^b}\qquad\textrm{and}\qquad
\xi=\xi(s,\e,\varrho)\ceq 1-\frac {y_1}{\e}.
\end{equation}

\begin{lemma}
	\label{prop:estimate_gc_1} 
	Fix $C>0$ and let $\e,\varrho$ be such that $0<\z\leq C \e^{3q-1}$. For $\e\to 0$, we have:
	\begin{enumerate}[(i)]
		\item\label{est_gc_item_1} $y_0=o(\e)$; 
            \item\label{est_gc_item_2}  $\xi=o(1)$, that is $y_1=\e+o(\e)$; 
		\item\label{est_gc_item_3} $\displaystyle \alpha = \frac{b(q-1)}{2} \xi^2 +o(\xi^2)$.
	\end{enumerate}
\end{lemma}

\begin{proof}
	Using \eqref{zeta_t_0} and $\z\leq C \e^{3q-1}$, it holds  
	\[
	0<\frac{y_0}{\e} = \frac {\z^{1/b} }{\e (q-1)^{1/b}}\leq \frac{C^{\frac1b}}{(q-1)^{\frac1b}}\e^{\frac12-\frac1b}.
	\]
	Then, item \ref{est_gc_item_1} follows because $b>2$. With the notation introduced in \eqref{alpha_xi},
	we rewrite \eqref{eq:t1} as
 	\begin{equation}
		\label{eq:1idt1}
		1-((1-\xi)^b+\a)^{1/2}=q(1-\xi)^{b-1}((1-\xi)^b+\a)^{-1/2}\xi.
	\end{equation}
	By Dini's Theorem, the equation above implicitly defines a function 
	$\alpha=\phi(\xi)$ in a neighborhood of $(\alpha,\xi)=(0,0)$, that 
	satisfies $\phi(0)=\phi'(0)=0$ and $\phi''(0)\neq 0$. An explicit 
	computation of the second derivative proves item \ref{est_gc_item_3}. 
	This argument also shows that $\xi\to0$ as $\e\to0$, which in turn implies item \ref{est_gc_item_2}.
\end{proof}

\begin{proof}[Proof of Proposition \ref{prop:sublevel_length_estimate}]
	Set $\hat y_0\ceq y_0(0,\e,\z)$, i.e., $\hat y_0$ is the $y$-coordinate of the tangency point between $\mathrm{spt}(\Gamma_\z)$ and the line segment starting from 0. Note that, by construction, $y_0=y_0(s,\e,\varrho)<\hat y_0<y_1(s,\e,\varrho)=y_1$. We decompose the length of $\nu_{s,\e}$ in the following sum:
	\[
	L(\nu_{s,\e}) = \underbrace{\left|\pi(\g(-s))-\Gamma_\z(y_0)\right| 
	+ L(\Gamma_\z |_{[y_0,\hat y_0]})}_{\ceq L_1} + \underbrace{L(\Gamma_\z |_{[\hat y_0,y_1]})}_{\ceq L_2}
	+ \underbrace{\left|\Gamma_\z(y_1)-\pi(\gamma(\e))\right|}_{\ceq L_3}.
	\]
    Firstly, we estimate $L_1$ by
    \begin{equation}
        \label{L1}
        L_1\geq \left|\pi(\g(-s))-\Gamma_\z(\hat y_0))\right| \geq s+ \hat y_0.
    \end{equation}
	Secondly, we estimate $L_2$.  
    On the one hand, thanks to the estimate \eqref{eq:length_abnormal_curve} on the length of $\g_{s,\e}$, we have
	\begin{equation} \label{L_2_prel}
		L_2 = \int_{\hat y_0}^{y_1} \sqrt{1 + q^ 2 t^{2b-2} (t^b+\z)^{-1}} dt
		\geq \int_{\hat y_0}^{y_1} \sqrt{1 + q^ 2 t^{b-2}  } dt =L(\g|_{[\hat y_0,y_1]}) .
	\end{equation}
    On the other hand, with the shorthand $\chi(\e) \ceq   \sqrt{1+q^2 \e^{b-2}} $, we have
	\begin{equation}\label{L2_ter} 
		L(\g|_{[y_1,\e]})  
		=\int_{y_1}^\e \sqrt{1+q^2 t^{b-2}}dt 
		= \e \int_{1-\xi}^1  \sqrt{1 +q^2 \e^{b-2} \tau^{b-2}} d\tau \leq  \e\chi(\e)\xi .
	\end{equation}
	Pairing \eqref{L_2_prel} and \eqref{L2_ter},
	we obtain that
	\begin{equation}\label{L_2}
		L_2\geq  L(\g|_{[\hat y_0,\e]}) -\e  \chi(\e)  \xi.
	\end{equation}
	Finally, we estimate $L_3$. We claim that there exists $C_1>0$ such that
	\begin{equation}\label{L_3}
		L_3\geq \e  \chi(\e)  \xi - C_1\e^{b-1}\xi^2 .
	\end{equation} 
	Recalling \eqref{eq:t1} and \eqref{alpha_xi}, we have that
	\begin{equation}\label{prop:pitagora_1}
	    L_3=\left|\Gamma_\z(y_1)-\pi(\gamma(\e))\right|=\sqrt{(\e-y_1)^2+\left(\e^q-(y_1^b+\z)^{1/2}\right)^2} = \e\xi \sqrt{1+m_1^2}.
	\end{equation}
	Then, using that $y_1 = \e(1-\xi)$ and $\alpha=\z/\e^b$, the quantity $m_1^2$ satisfies
	\begin{equation}\label{slope}
		\begin{split}
			m_1^2  &= q^ 2 t_1^{2b-2} \big( y_1^b+\z\big) ^{-\frac12}
			= q^2 \e^{2b-2} (1-\xi) ^{2b-2} \big( \e^b (1-\xi)^b +\z\big)^{-\frac12}
			\\
            &= q^ 2 \e^{b-2} (1-\xi)^{2b-2} \big( (1-\xi)^b+\a)^{-\frac12}
			\\&
			= q^ 2 \e^{b-2}\big ( 1 -2(q-1) \xi +o(\xi)\big),
		\end{split}
	\end{equation}
	where in the last line we used Lemma \ref{prop:estimate_gc_1}\ref{est_gc_item_3}, for $\alpha=\a(\xi)$. Plugging \eqref{slope} in the expression \eqref{prop:pitagora_1} and recalling that $\xi=o(\e)$ as $\e\to 0$, by Lemma \ref{prop:estimate_gc_1}\ref{est_gc_item_2}, it follows that
	\[
	L_3 = \e\xi\chi(\e)+ \e\xi \frac{ m_1^ 2 - q^ 2 \e^{b-2}}{\sqrt{1+m_1^2}+ \chi(\e) }=\e\xi\chi(\e)- q^ 2 \e^{b-1}\xi^2 \frac{ 2(q-1)  +o(1)}{\sqrt{1+m_1^2}+ \chi(\e) }\geq \e\xi\chi(\e)- C_1\e^{b-1}\xi^2,
	\]
    for some constant $C_1>0$, provided that $\e_0$ is sufficiently small. This proves \eqref{L_3}.

	All in all, putting together the inequalities 
	\eqref{L1}, \eqref{L_2}, and \eqref{L_3}, we obtain
	\begin{equation}
	    \begin{split}
	        L(\nu_{s,\e}(\cdot,\z))&\geq s+\hat y_0 + L(\g|_{[\hat y_0,\e]}) - C_1\e^{b-1}\xi^2=L(\gse) + \hat y_0 - L(\g|_{[0,\hat y_0]}) - C_1\e^{b-1}\xi^2.
	    \end{split}
	\end{equation}
    By \eqref{eq:length_abnormal_curve},  $L(\g|_{[0,\e]})=\e + \frac{q^2}{2(b-1)}\e^{b-1} + o(\e^{b-1})$, hence, by Lemma \ref{prop:estimate_gc_1}\ref{est_gc_item_1}, up to shrinking $\e_0$, we find a constant $C_2>0$ such that $\hat y_0- L(\g|_{[0,\hat y_0]})\geq -C_2\hat y_0^{b-1}$. Thus, we have 
    \begin{equation}
        \label{L123}
        L(\nu_{s,\e}(\cdot,\z))\geq L(\gse) -C_2 \hat y_0^{b-1} - C_1\e^{b-1}\xi^2 = L(\gse) -C_2 C_3 \z^{1-\frac1b} - C_1\e^{b-1}\xi^2,
    \end{equation}
    where, in the last equality, we used that $\hat y_0^{b-1} = C_3 \z^{1-\frac1b}$ for some $C_3>0$, by \eqref{zeta_t_0}. We claim that 
	\begin{equation}\label{LAST}
		\e^{b-1}\xi^2
		\leq\z^{1-\frac1b}.
	\end{equation}
    This concludes the proof as inequality \eqref{L123} implies \eqref{STIMA}, with $K=C_2C_3+C_1$. To prove \eqref{LAST}, observe that, by Lemma \ref{prop:estimate_gc_1}\ref{est_gc_item_3}, $\e^{b-1}\xi^2\leq C_4 \e^{b-1}\a = C_4 \z \e^{-1}$ for some constant $C_4>0$. Using the assumption on $\varrho$, we see that $\varrho\leq C\e^{3q-1}\leq \e^b$ since $b>2$, or equivalently, $\z\e^{-1}\leq \z^{1-\frac1b}$. Therefore, $\e^{b-1}\xi^2\leq C_4 \z \e^{-1} \leq \z^{1-\frac1b}$, proving \eqref{LAST} and concluding the proof. 
\end{proof}

A direct consequence of Proposition \ref{prop:sublevel_length_estimate} is a rough upper bound on the length of loops of optimal competitors.

\begin{corollary}
    \label{cor:rough_bound_loops}
    There exist $K,\e_0>0$ such that, for all $(\e,s)\in\mc I_{\e_0}$, all $\w\in \CO$, and every loop $\ell$ of $\w$, it holds
    \begin{equation}
        L(\ell)\leq K\widetilde\beta ^{1-\frac1b}.
    \end{equation}
\end{corollary}

\begin{proof}
    Let $K,\e_0>0$ be given by Proposition \ref{prop:sublevel_length_estimate} with $C=1$. Up to shrinking $\e_0$ if needed, we may assume $\e_0<1$. Fix $(\e,s)\in\I_{\e_0}$, $\w\in\CO$, and assume that $\w$ has a loop $\ell$. Let $J_\ell=[s_\ell^-,s_\ell^+]$ be loop interval of $\ell$ and define the curve $\nu:[0,1]\to\R^2$ as the constant speed reparametrization of $\omega|_{[0,s_\ell^-]}\ast \omega|_{[s_\ell^+,L(\omega)]}$. By construction, we have
    \begin{equation}
        \nu(0)=\pi(\gamma(-s)),\quad \nu(1)=\pi(\gamma(\e))\quad\text{and}\quad \spt(\nu)\subset D_{\wt\beta} 
    \end{equation}
    If $\nu_{s,\e}$ is the solution to the minimization problem \eqref{MIN} with $\z=\wt\beta$, then we have $L(\nu)\geq L(\nu_{s,\e})$. Since $\wt\beta<\e^M<\e^{3q-1}$ by Corollary \ref{cor:newLS} (recall that $M>3q-1$ fixed), we can now apply Proposition \ref{prop:sublevel_length_estimate} with $C=1$, and we get
    \begin{equation}
        L(\nu)\geq L(\g_{s,\e})-K\widetilde\beta ^{1-\frac1b}.
    \end{equation}
    Moreover, by length-minimality of $\w$, we have $L(\w)=L(\ell)+L(\nu)\leq L(\gamma_{s,\e})$. Then, we deduce 
    \begin{equation}
        L(\ell)\leq L(\gamma_{s,\e})-L(\nu)\leq K\widetilde \beta^{1-\frac1b}. \qedhere
    \end{equation}
\end{proof}

\section{Geometric obstructions to optimality}
\label{sec:geom_obs}

In this section, for the reader's convenience, we report some results on the geometry of plane curves and some obstructions to optimality. 

\subsection{Winding number of plane curves and weighted area}

Let $\eta:[0,\tau]\rightarrow \R^2$ be a closed continuous curve. For any $y\not\in\mathrm{spt}(\eta)$, the winding number of $\eta$ around $y\in\mathbb{R}^2$ is defined as
\begin{equation*}
\mathrm{ind}(\eta,y)\ceq \frac{1}{2\pi i}\oint_{\eta}\frac{d\zeta}{\zeta-y},
\end{equation*}
under the usual identification $\mathbb{R}^2\cong\mathbb{C}$. The function $y\mapsto\mathrm{ind}(\eta,y)$ is holomorphic outside $\spt(\eta)=\eta([0,\tau])$, and integer-valued. For every $k\in\mathbb{Z}$, we define
\begin{equation*}
    \E_k(\eta)\ceq \{y\not\in\mathrm{spt}(\eta)\,|\,\mathrm{ind}(\eta,y)=k\}\qquad\text{and}\qquad \E(\eta)\ceq \bigcup_{k\in\Z\setminus\{0\}}\E_k(\eta),
\end{equation*}
 so that the sets $\E_k(\eta),\,\E(\eta)\subset\mathbb{R}^2$ are open and $\E(\eta)$ is bounded. In addition, if $\eta$ is simple, by the Jordan Curve Theorem, it divides the plane into two disjoint open connected components. Between those two components, we define $\D(\eta)$ to be the bounded one. In this case, we have
\begin{equation}\label{eq:jordan_oriented}
\mathcal{D}(\eta) = \mathcal{E}_{\sigma}(\eta) = \mathcal{E}(\eta),
\end{equation}
where $\sigma\in\{-1,1\}$. We say that the curve $\eta$ is positively (resp. negatively) oriented if $\sigma=1$ (resp. $\sigma=-1$) and $\sigma$ is called the orientation of the curve.

We define the weighted area of a closed continuous curve $\eta:[0,\tau]\to\R^2$ as
\[
A(\eta) \ceq  \sum_{k\in \Z} k \, \mathcal{L}_Q\left(\mathcal{E}_k(\eta)\right),
\]
where
\[
\mathcal{L}_Q\left(\mathcal{E}_k(\eta)\right) \ceq  \int_{\mathcal{E}_k(\eta)} Q(x) \, dx \qquad \forall k \in \Z.
\]
A straightforward application of Stokes' Theorem shows that, for a closed curve $\eta$,
\begin{equation}
    A(\eta)= \int_\eta P(x)^2dx_2 = \int_{I_\eta} \dot\eta_2(t) P(\eta(t))^2dt.
\end{equation}

\begin{remark}
\label{rmk:stoked}
    For every $s,\e>0$ and $\w\in\CC$, we may consider the closed curve $\eta\ceq \w\ast \pr(\check{\gamma}_{s,\e})$, where $\check{\gamma}_{s,\e}$ is the inverse reparametrization of $\gse$. Then, by Definition \ref{def:competitor}\ref{item:def_competitor_ii} and the Stokes' theorem, we see that $A(\eta)=0$.       
\end{remark}

Finally, we recall a useful inequality on the weighted area of a closed continuous curve $\eta$, which is a consequence of Rad\'o's isoperimetric inequality, cf.\ \cite{Rado}, namely 
    \begin{equation}
    \label{eq:rado}
        |A(\eta)| \leq \frac{L(\eta)^2}{4\pi}\sup_{x\in \mc E(\eta)} |Q(x)|. 
    \end{equation}

\subsection{Geometry of plane curves}
For any pair of linearly independent vectors $v,w\in\R^2$, denote by $\angle(v,w)\in (-\pi,\pi)$ the oriented angle between them. If $v=\mu w$, for some $\mu>0$, we set $\angle(v,w)\ceq 0$. Note that, with this definition, $\angle(v,w) = - \angle(w,v)$. Moreover, we set
\begin{equation}
    \label{eq:tangenti}
    \begin{split}
        \mbox{Tan}^{1}(v,w)&\ceq \big\{u\in\R^2 \mid \angle(v,u)\in \big(\angle(v,w),\pi\big)\big\},
    \\
    \mbox{Tan}^{-1}(v,w)&\ceq \big\{u\in\R^2 \mid \angle(v,u)\in \big(-\pi,\angle(v,w)\big)\big\}.
    \end{split}
\end{equation}
Let $\eta\in \AC([0,\tau],\R^2)$ be a closed, simple curve, parametrized by arc-length. We say that $\eta$ is piecewise smooth if there are $0=t_0 < t_1 < \ldots <t_N = \tau$ such that $\eta|_{ [t_i,t_{i+1}]}$ is smooth for all $i=0,\ldots,N-1$. For every $i$, we denote by $\dot{\eta}(t_i^-)$ (resp.\ $\dot{\eta}(t_i^+)$) the left (resp.\ right) derivative of $\eta$ at $t_i$, with the convention that $\dot{\eta}(t_0^-)\ceq \dot{\eta}(t_N^-)$ and $\dot{\eta}(t_N^+)\ceq \dot{\eta}(t_0^+)$. Moreover, we set $\delta_i \ceq  \angle(\dot{\eta}(t_i^-),\dot{\eta}(t_i^+))$.

\begin{remark}
    \label{rem:tangenti_estremi}
    We adopt the following conventions for angles and tangents when the curve $\eta$, with orientation $\s$, has a cusp at $t_i$, i.e., when $\dot{\eta}(t_i^-) = -\dot{\eta}(t_i^+)$:
    \begin{enumerate}[(i)]
        \item if the cusp points outside $D(\eta)$, i.e., if $r\dot\eta(t_i^-)\notin\D(\eta)$ for all $r>0$, then we set
        \begin{equation}
            \delta_i\ceq \s\pi, \quad  \Tan^\s(\dot\eta(t_i^-),\dot\eta(t_i^+))\ceq \emptyset, 
            \quad \text{and} \quad
            \Tan^{-\s}(\dot\eta(t_i^-),\dot\eta(t_i^+))\ceq \R^2\setminus \{r\dot\eta(t_i^+)\mid r>0\}.
        \end{equation}

        \item if the cusp points inside $D(\eta)$, i.e., if $r\dot\eta(t_i^-)\in\D(\eta)$ for small $r>0$, then we set
        \begin{equation}
            \delta_i\ceq -\s\pi, \quad  \Tan^\s(\dot\eta(t_i^-),\dot\eta(t_i^+))\ceq \R^2\setminus \{r\dot\eta(t_i^+)\mid r>0\}, 
            \quad \text{and} \quad
            \Tan^{-\s}(\dot\eta(t_i^-),\dot\eta(t_i^+))\ceq \emptyset.
        \end{equation}
    \end{enumerate}  
    Note that, in case (i) we have $\s\de_i>0$, while in case (ii) we have $\s\de_i<0$, independently of $\s$.
\end{remark}

Finally, since $\eta:[0,\tau]\to\R^2$ is parametrized by arc-length, there exists a piecewise smooth function $\alpha : [0,\tau] \rightarrow \R$ satisfying $\dot{\eta}(t)=(\cos\alpha(t),\sin \alpha(t))$ for all $t\in [0,\tau]\setminus\{t_0,\ldots,t_N\}$. We can define the signed curvature of $\eta$ as
\begin{equation}\label{eq:signed_curvature}
\kappa(t) = \dot{\alpha}(t) \qquad \forall t \in [0,\tau]\setminus \{t_0,\ldots,t_N\}.
\end{equation}

\begin{remark}[Gauss-Bonnet for plane curves]
\label{rmk:gauss-bonnet}

If $\eta$ has orientation $\sigma\in\{-1,1\}$, then, the Gauss-Bonnet Theorem implies that
    \begin{equation}
        \label{eq:gauss-bonnet-thm}
        \sum_{i=0}^{N-1}\int_{t_i}^{t_{i+1}}\sigma \kappa(t)dt + \sum_{i=0}^{N-1}\sigma\delta_i = 2\pi.
    \end{equation}
    Viceversa, if Gauss-Bonnet Theorem \eqref{eq:gauss-bonnet-thm} holds with $\sigma\in\{-1,1\}$, then $\sigma$ is the orientation of $\eta$.
\end{remark}

We collect now some geometric lemmas on plane curves. We refer to \cite{notall} for the proof of Lemmas \ref{lem:enclosed_directions} and \ref{lem:gauss-bonnet_nerfed}. 

\begin{lemma}
\label{lem:enclosed_directions}
Let $\eta:[0,\tau] \rightarrow \R^2$ be a piecewise smooth, closed, simple, and arc-length parametrized curve, and let $\sigma\in\{-1,1\}$ be its orientation, cf.\ \eqref{eq:jordan_oriented}. Assume that $\angle(\dot\eta(t^-),\dot\eta(t^+))\neq \s\pi$, cf.\ Remark \ref{rem:tangenti_estremi}.
Then, for any $t\in [0,\tau)$, there exists $u\in \mbox{\rm Tan}^\sigma(\dot{\eta}(t^-),\dot{\eta}(t^+))$ such that 
\[
\eta(t) + r u \in  \mathcal{E}(\eta), \quad \forall r\in(0,1). 
\]
\end{lemma}

\begin{lemma}
\label{lem:gauss-bonnet_nerfed}
Let $\eta:[0,\tau] \rightarrow \R^2$ as above and fix $\sigma\in\{-1,1\}$. Then, it holds:
\begin{enumerate}[(i)]
\item\label{item:gb1} if $\sigma\kappa|_{ [t_i,t_{i+1}]}\geq 0$, and $\sigma\delta_i\in  [0,\pi]$, for every $i=0, \ldots, N-1$, then $\mathcal{D}(\eta)$ is convex, $\sigma$ is the orientation of $\eta$, and $\sigma\delta_i \in [0,\pi)$, for every $i=0,\ldots, N-1$;
\item\label{item:gb2} if there are $i_1,i_2\in\{0,\ldots, N-1\}$, with $i_1\neq i_2$, such that
\[
\sum_{i=0}^{N-1} \int_{t_i}^{t_{i+1}} \sigma\kappa(t)\, dt \geq 0 \quad \mbox{and} \qquad \sigma\delta_i \in [0,\pi], \quad \forall i \in \{0,\ldots, N-1\} \setminus \{i_1,i_2\},
\]
then $\sigma$ is the orientation of $\eta$ and, in particular, $\mathcal{D}(\eta) = \mathcal{E}_{\sigma}(\eta)$.
\end{enumerate}
\end{lemma}

\begin{lemma}
    \label{lem:convexcurvature}
	Let $\eta:[0,\tau]\to\R^2$ be a smooth, simple and closed curve, parametrized by arc-length. Assume that its curvature $\k(t)$ has constant sign $\sigma\in\{-1,1\}$ and that the angle $\delta\ceq \angle(\dot\eta(\tau^-),\dot\eta(0^+))$ satisfies $\sigma\delta\in[0,\pi)$. Then, there exists $t_*\in(0,\tau)$ such that $|\kappa(t_*)|\geq \pi\,\tau^{-1}=\pi L(\eta)^{-1}$.
\end{lemma}

\begin{proof}
    By the Gauss-Bonnet Theorem \eqref{eq:gauss-bonnet-thm} applied to the smooth curve $\eta$, we have
    \begin{equation}
        \int_0^\tau \sigma \kappa(t)dt = 2\pi - \sigma\delta > \pi.  
    \end{equation}
    Hence, by the mean value theorem, there exists $t^*\in [0,\tau]$ such that  
    \begin{equation}
        |\kappa(t^*)|=\sigma\kappa(t^*)=\frac1\tau\int_0^\tau \sigma \kappa(t)dt > \frac\pi\tau.  \qedhere
    \end{equation}
\end{proof}

\subsection{Comparison of curves through cut-and-paste arguments}
In Lemma \ref{lem:w1>0}, we established that normal competitors live in the region $\{x_1\geq 0\}$. Starting from this observation, we can derive some obstructions to optimality. These obstructions are purely local in nature and do not depend on the endpoints of the curve. Thus, the proofs of the following two lemmas can be obtained repeating verbatim the proof of \cite[Lem.\ 3.7]{notall} and are therefore omitted.

\begin{lemma}\label{lem:cut_and_paste1}
    For every $\e,s>0$, and every $\w\in\mc C_{s,\e}$, we have that $\omega\not\in\mathcal{C}_{s,\e}^{\rm{opt}}$ as soon as one of the following conditions is satisfied:
    \begin{enumerate}[(i)]
        \item\label{c_a_p_1_item_i} There exist a loop $\ell$ of $\omega$ and a Lipschitz closed curve $\eta:[0,\tau]\to\mathbb{R}^2$ that intersects the set $\overline{\spt(\w)\setminus\spt(\ell)}$, such that $L(\eta)\leq L(\ell)$ and $|A(\ell)|\leq|A(\eta)|$;
        \item\label{c_a_p_1_item_ii} There exist $0\leq t_1<t_2\leq L(\omega)$ and a Lipschitz closed curve $\eta:[0,\tau]\to\mathbb{R}^2$ that intersects the curve $\omega|_{[0,t_1]}*[\omega(t_1),\omega(t_2)]*\omega|_{[t_2,L(\omega)]}$, such that $L(\eta)\leq L(\omega|_{[t_1,t_2]})-L([\omega(t_1),\omega(t_2)])$ and $|A(\omega|_{[t_1,t_2]}*[\omega(t_1),\omega(t_2)])|<|A(\eta)|$;
        \item\label{c_a_p_1_item_iii} There exist a simple loop $\ell$ of $\omega$, $t^{*}\in I_\omega\setminus{\rm int}(J_\ell)$ and $k\in\{-1,1\}$, such that $k\cdot (P\circ\ell)>0$ and $\{\omega(t^{*})+k(0,s)|s>0\}\cap\rm{spt}(\ell)\neq\emptyset$;
        \item\label{c_a_p_1_item_iv} There exist $0\leq t_1<t_2\leq L(\omega)$ such that $(P\circ\omega)|_{(t_1,t_2)}<0$, $P(\omega(t_1))=P(\omega(t_2))=0$, and $\omega$ is injective on $[t_1,t_2)$;
        \item\label{c_a_p_1_item_v} There exist a simple loop $\ell$ of $\omega$ and $t^{*}\in I_\omega\setminus {\rm int}(J_\ell)$ such that $\max_{t\in J_\ell}|Q(\omega(t))|\leq Q(t^{*})$;
        \item\label{c_a_p_1_item_vi} There exist two loops $\ell_1$ and $\ell_2$ of $\omega$, such that ${\rm int}(J_{\ell_1})\cap{\rm int}(J_{\ell_2})=\emptyset$ and $Q\circ\ell_1,Q\circ\ell_2\geq0$.
    \end{enumerate}
\end{lemma}
In addition, the next lemma provides two more obstructions to optimality, that give bounds on the behavior of an optimal competitor in the region close to the endpoint.
\begin{lemma}\label{lem:cut_and_paste2}
    For every $K>0$, there exists $\e_0>0$ such that for all $0<\e<\e_0$, $s>0$, and $\w\in \CC$, we have that $\w\notin \CO$ as soon as one of the following conditions is satisfied: 
    \begin{enumerate}[(i)]
        \item\label{c_a_p_2_item_i} There exist a loop $\ell$ of $\omega$ and $t^{*}\in I_\omega\setminus{\rm int}(J_{\ell})$ such that $\omega_2(t^{*})>K\e$, $Q(\w(t^*))<0$, and $\max_{t\in J_\ell}|Q(\omega(t))|\leq-Q(\w(t^*))$;
        \item\label{c_a_p_2_item_ii} There are two loops $\ell_1$ and $\ell_2$ of $\omega$ such that $\omega_2(s_{\ell_1}^{-}),\omega_2(s_{\ell_2}^{-})\geq K\e$ and ${\rm int}(J_{\ell_1})\cap{\rm int}(J_{\ell_2})=\emptyset$.
    \end{enumerate}
\end{lemma}

\section{Loops of optimal competitors}
\label{sec:looooooooops}

This section is devoted to an in-depth study of loops of optimal competitors. In particular, we show that any optimal competitor must have self-intersections, and thus, loops, cf.\ Definition \ref{def:loops}. 

Given $s,\e>0$, let $\w\in\mathcal{C}_{s,\e}^{\rm{opt}}$.  Since $\omega$ is real-analytic, the set $\mr{spt}(\w)\cap\mr{spt}(\g)\subset \{\wt P=0\}$ is finite. We index the points of intersection as follows: we set $N=N(\omega)\ceq \# \big(\mr{spt}(\w)\cap\mr{spt}(\g)\big)-1$ and we define the unique times $\tau_i=\tau_i(\w)$, $i=0,\dots,N$, such that
\begin{equation}
	0=\tau_0<\tau_1<\dots<\tau_N=L(\w) \quad \mr{and} \quad \widetilde{P}(\tau_i)=0, \text{ for all } i=0,\dots,N.
\end{equation}

\begin{remark}
\label{rmk:relative_position_first_times}
    By construction, the intervals $(\tau_i,\tau_{i+1})\subset [0,L(\w)]$  are the maximal intervals where the functions $t\mapsto \widetilde{P}(\w(t))$ and $t\mapsto Q(\w(t))$ have constant sign. Moreover, note that by Lemma \ref{lem:w1>0}, for every $1\leq i\leq N$, we have $\w(\tau_i)\in \{\widetilde P=0, x_1>0\}$ and thus, $\w_2(\tau_i)> 0 > \w_2(\tau_0)=-s$.  
\end{remark}

\begin{definition}
	\label{def:tau_i}
    Given $s,\e>0$ and $\w\in\CO$, we set $\I=\I(\w)\ceq \{0,\dots,N(\w)-1\}$ and $I_i=I_i(\w)\ceq [\tau_i,\tau_{i+1}]$, for $ i\in\I$. We say that:
\begin{enumerate}[(i)]
	\item $I_i$ is {\em positive}, and $i$ is referred to as a {\em positive index}, if $\widetilde{P}\circ\w|_{I_i}\geq0$;
	\item $I_i$ is {\em negative}, and $i$ is referred to as a {\em negative index}, if $\widetilde{P}\circ\w|_{I_i}\leq0$.
\end{enumerate}
Finally, we write $\I_+$ (resp. $\I_-$) for the set of positive (resp. negative) indices.
\end{definition}

Observe that the maps $I_i\ni t\mapsto \widetilde{P}(\w(t)), Q(\w(t))$ are positive if $i\in \I_+$ and negative if $i\in \I_-$.

\begin{definition}
	\label{def:1simpleloop}
	Given $s,\e>0$ and $\w\in\CO$, let $\ell$ be a loop of $\w$, with associated interval $J_\ell=[s_\ell^-,s_\ell^+]\subset I_\w$. We say that:
    \begin{itemize}
        \item[(i)] $\ell$ is the \emph{first simple loop} of $\omega$ if $\omega|_{[0,s_\ell^{+})}$ is injective;
        \item[(ii)] $\ll$ is {\em positive} (resp., {\em negative}) if ${J_\ll}\subset I_{i}$ for some $i\in\I_+$ (resp., $i\in\I_-$). 
    \end{itemize}
\end{definition}

\begin{lemma}
\label{lem:preliminary}
There exists $\e_0>0$ such that, for every $0<\e<\e_0$, $s>0$, and $\w\in \CO$, we have
 \begin{enumerate}[(i)]
        \item\label{lem_prelim_item_ii} $\omega$ is not injective;
        \item\label{lem_prelim_item_iii} the covector $\lambda_\w$ associated to $\omega$ via Lemma \ref{lem_polarnormal} satisfies $\lambda_\w\neq 0$;
        \item\label{lem_prelim_item_i} $-2\e-s<\omega_2(t)<2\e$, for every $t\in I_\omega$.

        \item\label{lem_prelim_item_i_bis} if, in addition, $(\e,s)\in \mc I_{\e_0}$, then we have $\omega_1(t)<\big(1+\frac{b}{\sqrt{b-1}}\big)\e^q$, for every $t\in I_\omega$.
\end{enumerate}  
\end{lemma}

\begin{proof}
     Proof of \ref{lem_prelim_item_ii}. Note that, thanks to Lemma \ref{lem:w1>0}, $0\in\I_+$. Assume by contradiction that $\w$ is injective. Then, by Lemma \ref{lem:cut_and_paste1}\ref{c_a_p_1_item_iv}, $\w$ cannot cross the curve $\{\widetilde P=0\}$, therefore, we have $\wt P\circ\w\geq0$ and $I_0=[0,L(\w)]$. Consider now the closed simple curve $\eta\ceq \w\ast\pr(\check{\g}_{s,\e})$, where $\check{\g}_{s,\e}$ denotes the inverse reparametrization of $\g_{s,\e}$ and, recall that, by Remark \ref{rmk:stoked}, $A(\eta)=0$. The open set $\D(\eta)$ is non-empty, since $\w$ does not coincide with the arc-length reparametrization of $\pr(\g_{s,\e})$, and it is contained in the region $\{Q>0\}$. 
     Hence, by Stokes' Theorem, it holds 
     \begin{equation}
         A(\eta)=\int_{I_\eta} \dot\eta_2(t)P^2(\eta(t))dt = \sigma\int_{\D(\eta)} Q(x,y) dxdy \neq 0, 
     \end{equation}
    where $\sigma\in \{-1,1\}$ is the orientation of $\eta$, giving a contradiction. 
    
    Proof of \ref{lem_prelim_item_iii}. Observe that, if $\lambda=0$, $\w$ is a straight line by Lemma \ref{lem_polarnormal}, and thus injective. This is not possible by item \ref{lem_prelim_item_ii}. 

     Proof of \ref{lem_prelim_item_i}. Firstly, recall that the two endpoints of $\w$ are $(0,-s)$ and $(\e^q,\e)$. We argue by contradiction, comparing $L(\w)$ with the length of the concatenation of two line segments $[(0,-s),\w(t)]\ast[\w(t),(\e^q,\e)]$, where $t\in I_\w$. Assume that there exists $t\in I_\w$ such that $\w_2(t)=2\e$. Then, it follows
     \begin{align}
         L(\w)&\geq L([(0,-s),\w(t)]\ast[\w(t),(\e^q,\e)])\\
         &=\sqrt{(\w_1(t))^2+(2\e+s)^2}+\sqrt{(\w_1(t)-\e^q)^2+\e^2}\\
         &\geq 2\e + s + \e,
     \end{align}
     which contradicts the minimality of $\w$, for $\e$ small enough, according to \eqref{eq:length_abnormal_curve}. This proves that $\w_2(t)<2\e$, for every $t\in I_\w$.
     Similarly, assuming that $\w_2(t)=-2\e-s$, for some $t\in I_\w$, we get
     \begin{equation}
         L(\w)\geq\sqrt{(\w_1(t))^2+(-2\e)^2}+\sqrt{(\w_1(t)-\e^q)^2+(-2\e-s-\e)^2}\geq 3\e + s,
     \end{equation}
     which is again in contradiction with \eqref{eq:length_abnormal_curve}, thus proving that $\w_2(t)>-2\e-s$, for every $t\in I_\w$.

    Proof of \ref{lem_prelim_item_i_bis}. Let $c\geq 1+\frac{b}{\sqrt{b-1}}$ and assume by contradiction that there exists $t\in I_\w$ such that $\w_1(t)=c\e^q$. Then, as for \ref{lem_prelim_item_i}, it follows that
     \begin{equation}\label{prelim_estimate_1}
            L(\w)\geq \sqrt{c^2\e^b+(\w_2(t)+s)^2}+\sqrt{(c-1)^2\e^b+(\w_2(t)-\e)^2}.
     \end{equation}
     The function $F(z)\ceq \sqrt{c^2\e^b+(z+s)^2}+\sqrt{(c-1)^2\e^b+(z-\e)^2}$ has a global minimum in $z_*=\frac{c\e-(c-1)s}{2c-1}$, and
     \begin{equation}
         F(z_*)=(s+\e)\sqrt{1+(2c-1)^2\frac{\e^b}{(s+\e)^2}}.
     \end{equation}
     Since $(\e,s)\in \mc I_{\e_0}$, we have $s< \e^2<\e$, and thus, also using the smallness of $\e$, it follows that
     \begin{equation}
        \label{eq:stima_nuova_prel_w}
        L(\w)\geq (\e+s)\sqrt{1+\frac{(2c-1)^2}{4}\e^{b-2}}
         \geq \e+s+\frac{(2c-1)^2}{16}\e^{b-1}.
     \end{equation}
     On the other hand, by \eqref{eq:length_abnormal_curve} and small $\e>0$ we have
     \begin{equation}
     \label{eq:stima_nuova_prel_gamma}
         L(\gse)\leq \e + s + \frac{b^2}{4(b-1)}\e^{b-1}.
     \end{equation}
     From \eqref{eq:stima_nuova_prel_w} and \eqref{eq:stima_nuova_prel_gamma} we reach a contradiction.
\end{proof}

\begin{lemma}[Localized estimates]
    \label{lem:optimal}
    Fix $C>0$. Then, we have:
    \begin{enumerate}[(i)]
        \item\label{lem:optimal_item_i} there exists $\e_0>0$ such that for all $0<\e<\e_0$, $s>0$, and $\w\in\CO$, we have
        \begin{equation}
            \w_2(t_0) \geq C\e \quad\implies\quad \w_2(t)\geq \frac C2\e, \quad \text{for all } t \in [t_0,L(\w)].
        \end{equation}
        \item\label{lem:optimal_item_ii} there exists $\e_0>0$ such that for all $(\e,s)\in\I_{\e_0}$ and $\w\in\CO$ we have
        \begin{equation}
            \w_2(t_0) \geq C\e \quad\implies\quad \w_1(t)\geq \left(\frac C4\e\right)^q, \quad \text{for all } t \in [t_0,L(\w)].
        \end{equation}
    \end{enumerate}
\end{lemma}

\begin{proof}
    Let $C>0$ and let $t_0 \in I_\w$ be such that $\w_2(t_0 )\geq C\e$. 
    
    Proof of \ref{lem:optimal_item_i}. Fix $\e_0>0$ such that, according to \eqref{eq:length_abnormal_curve}, we have
    \begin{equation}
        \label{eq:1.1adattata}
        L(\gse)\leq s+\e+\frac C4\e, \qquad \text{for all } 0<\e<\e_0.
    \end{equation}
    For $0<\e<\e_0$ and $s>0$, fix $\w\in\CO$, and assume by contradiction that $\w_2(t)<\frac C2\e$, for some $t\in(t_0,L(\w)]$. Then, recalling that $\w_2(0)=-s$ and $\w_2(L(\w))=\e$, we deduce that
    \begin{equation}
        L(\w)\geq L(\w|_{[0,t_0]})+L(\w|_{[t,L(\w)]}) \geq |\w_2(t_0)-\w_2(0)| + |\w_2(L(\w))-\w_2(t)| > s+\e+ \frac C2\e, 
    \end{equation}
    which is in contradiction with \eqref{eq:1.1adattata}. 
    
    Proof of \ref{lem:optimal_item_ii}. Let $\e_0>0$ be smaller than the one found in (i) and such that 
    \[
    \left(\frac C2\e\right)^b - \e^M \geq \left(\frac C4\e\right)^b, \qquad \text{for all } 0<\e<\e_0,
    \]
    where $M>b$ is fixed in Notation \ref{notation:the}. Take $(\e,s)\in\I_{\e_0}$, and $\w\in\CO$. By Corollary \ref{cor:newLS}, we have $\wt\beta(\w)<\e^M$. Since $\w_2(t)\geq \frac C2\e$ for all $t\geq t_0$ by item \ref{lem:optimal_item_i},  and $P(x)=\widetilde P(x)$ on $\{x_2\geq 0\}$, we deduce that $P(\w(t))=\wt P(\w(t))$ for all $t\geq t_0$. Therefore, we obtain that, for every $t\geq t_0$: 
    \begin{equation}
        \w_1(t)^2 =  \widetilde P(\w(t)) + \w_2(t)^b \geq -\widetilde\beta  +\left(\frac C2\e\right)^b \geq - \e^M + \left(\frac C2\e\right)^b \geq \left(\frac C4\e\right)^b,
    \end{equation}
    which complete the proof. 
\end{proof}

We conclude this section by studying the winding of optimal competitors. Recall that, for a normal competitor $\omega\in\mathcal{C}_{s,\e}$, the sign of $t\mapsto Q(\w(t))$ is constant on $I_i$, for all $i\in \I$. Thus, by \eqref{eq:normalpolar}, the signed curvature of $\w|_{I_i}$ has constant sign on $I_i$.

\begin{lemma}
	\label{lem:fortissimo} 	
    For every $\e,s>0$ and every $\w\in\CO$ we have
    \begin{enumerate}[(i)]
        \item\label{item:lem_fortissimo_i} if $\w|_{I_i}$ admits a first simple loop $\ell$, and $\sigma\ceq  \sgn(\dot\theta|_{(\tau_i,\tau_{i+1})})$, then we have that $\D(\ell)$ is strictly convex, $\sigma$ is the orientation of $\ell$, $\sigma\cdot \angle(\dot\w(s_\ell^+),\dot\w(s_\ell^-))\in (0,\pi)$, and $\int_{J_\ell}|\dot\theta(t)|d t\in (\pi,2\pi)$;
        \item\label{item:lem_fortissimo_ii} if $i\in\I_+$, then $\w|_{I_i}$ admits at most one simple loop. Moreover, there is at most one index $i\in\I_+$ such that $\w|_{I_i}$ admits a loop; 
    \end{enumerate}
    Moreover, for every $K>0$ there is $\e_0>0$ such that for every $\e<\e_0$, $s>0$, and $\w\in\CO$, we have        
    \begin{enumerate}[(i),resume]
        \item\label{item:lem_fortissimo_iii} if $i\in\I_-$, then $\w|_{I_i}$ admits at least one simple loop $\ell$, with associated interval $J_\ell\subset(\tau_i,\tau_{i+1})$. If, in addition, the first simple loop of $\w|_{I_i}$ satisfies $\w_2(s_\ell^-)\geq K\e$, then $\ell$ is the unique loop of $\w|_{I_i}$.
    \end{enumerate}
\end{lemma}

\begin{proof}
    Proof of \ref{item:lem_fortissimo_i}. \ Observe that, by Lemma \ref{lem_polarnormal}, we must have $\dot\w(s_\ell^+)\neq\dot\w(s_\ell^-)$. Otherwise, by uniqueness of solution for the normal equation, $\w$ must be a closed integral curve, contradicting the end-point conditions. Hence, denoting by $\delta \ceq  \angle(\dot\w(s_\ell^+),\dot\w(s_\ell^-))$, we have $\delta\neq 0$. 
     
     We claim that $\sigma\delta>0$. By contradiction, assume that $\sigma\delta<0$. Then, by the definition of angle and smoothness of $\w$, 
     it follows that there exist $r_1,r_2>0$ such that 
     \begin{equation}
     \label{eq:conditions_on_Dl}
     \begin{split}
        &\w(s_\ell^--r)\in \D(\ell),\qquad\forall r\in(0,r_1)\quad\text{ and }\quad\w(s_\ell^--r_1)\in \partial\D(\ell);\\
        &\w(s_\ell^++r)\in \D(\ell),\qquad\forall r\in(0,r_2)\quad\text{ and }\quad\w(s_\ell^++r_2)\in \partial\D(\ell).
    \end{split}
    \end{equation}
    There are two cases: either $J_\ell\subsetneq I_i$, or $J_\ell=I_i$. In the first case, we have $s_\ell^--r_1>\tau_i$, and thus $\w$ has a loop $\ell'\neq \ell$ with associated interval $J_{\ell'}=[s_{\ell'}^-,s_{\ell'}^+]$, where $s_{\ell'}^-=s_{\ell}^--r_1$ and $s_{\ell'}^+\leq s_\ell^+$. This contradicts the fact that $\ell$ is the first loop of $\w|_{I_i}$. If instead $J_\ell=I_i$, we must have $\w(s_\ell^-)=\w(s_\ell^+)\in\spt(\pr(\gamma_\e))$. Note that Lemma \ref{lem:enclosed_directions} implies that 
    \begin{equation}
    \label{eq:uffa1}
    \spt(\ell)\cap\{\w(s_\ell^-)+\mu v \mid \mu>0\}\neq\emptyset, \qquad \text{for all } v\in \Tan^\s\big(\dot\w(s_\ell^+),\dot\w(s_\ell^-)\big).
    \end{equation}
    Moreover, by $\s\de<0$ and by the convexity of $\{P=0\}$, we see that 
    \begin{equation}
        \label{eq:uffa2}
        \Tan^\s\big(\dot\w(s_\ell^+),\dot\w(s_\ell^-)\big)\cap\{P>0\}\neq\emptyset \qquad \text{and} \qquad \Tan^\s\big(\dot\w(s_\ell^+),\dot\w(s_\ell^-)\big)\cap\{P<0\}\neq\emptyset.
    \end{equation}
    Equations \eqref{eq:uffa1} and \eqref{eq:uffa2} imply that there is $t\in(s_\ell^-,s_\ell^+)$ such that $P(\w(t))<0$, which is in contradiction with $J_\ell=I_i$. \ The claim $\s\de>0$ is proved, i.e., $\s\de\in(0,\pi]$.

    We now apply Lemma \ref{lem:gauss-bonnet_nerfed}\ref{item:gb1} to $\ell$, which ensures that $\sigma\delta\in (0,\pi)$, $\D(\ell)$ is convex and $\sigma$ is the orientation of $\ell$. Note that the signed curvature of $\ell$ is always non-zero, hence $\D(\ell)$ is strictly convex. By the Gauss-Bonnet Theorem, we also conclude that $\int_{J_\ell}\sigma\dot\theta(t)dt=\int_{J_\ell}|\dot\theta(t)|dt \in (\pi,2\pi)$.

    Proof of \ref{item:lem_fortissimo_ii}. \ First of all, by Lemma \ref{lem:cut_and_paste1}\ref{c_a_p_1_item_vi}, there is at most one index $i\in\I_+$ such that $\w|_{I_i}$ admits a loop. Let $\ell_1$ be the first simple loop of $\w|_{I_i}$ and assume by contradiction that $\w|_{I_i}$ has at least two self-intersections. This means that $\w|_{(s_{\ell_1}^+,\tau_{i+1}]}$ is not injective. Then, we define $s_{\ell_2}^+\ceq \min\{t>s_{\ell_1}^+\mid \w(t)\in \spt(\w|_{I_i})\}$ and  $s_{\ell_2}^-\in[\tau_i,s_{\ell_2}^+)$ to be such that $\w(s_{\ell_2}^-)=\w(s_{\ell_2}^+)$. Thus, $\ell_2\ceq \w|_{[s_{\ell_2}^-,s_{\ell_2}^+]}$ is a distinct loop from $\ell_1$ and, in addition, by Lemma \ref{lem:cut_and_paste1}\ref{c_a_p_1_item_vi}, we have $s_{\ell_2}^- < s_{\ell_1}^+$. Thus, we have the two cases: either $s_{\ell_2}^-\in[\tau_i,s_{\ell_1}^-]$ or $s_{\ell_2}^-\in (s_{\ell_1}^-,s_{\ell_1}^+)$.
    
    \noindent\textsc{Case 1: $s_{\ell_2}^-\in[\tau_i,s_{\ell_1}^-]$.} \ We define the curve $\ell:[s_{\ell_2}^-, s_{\ell_1}^- + s_{\ell_2}^+ - s_{\ell_1}^+]\to \R^2$ as
    \begin{equation}
        \ell(t)\ceq 
        \begin{cases}
            \w(t), \quad &t\in[s_{\ell_2}^-, s_{\ell_1}^-],\\
            \w(t+|J_{\ell_1}|), \quad & t\in[s_{\ell_1}^-, s_{\ell_2}^+ - |J_{\ell_1}|].
        \end{cases}
    \end{equation}
     The curve $\ell$ is closed, simple, piecewise smooth, parametrized by arc-length, and has signed curvature of sign $\sigma$ (as for $\ell_1$). On the one hand, since $\ell$ has at most two singularities, by Lemma \ref{lem:gauss-bonnet_nerfed}\ref{item:gb2}, $\sigma$ is the orientation of $\ell$, and $\D(\ell)=\E_\sigma(\ell)$. On the other hand, note that, by construction, we have
     \begin{equation}
         \label{eq:derivate_improbabili}
         \dot\w(s_{\ell_1}^-)=
         \dot\ell\big((s_{\ell_1}^-)^-\big) 
         \quad \text{and} \quad 
         \dot\w(s_{\ell_1}^+)=
         \dot\ell\big((s_{\ell_1}^-)^+\big),
     \end{equation}
    which implies that 
    \begin{equation}
        \label{eq:tangenti_improbabili}
        \Tan^\s(\dot\w(s_{\ell_1}^+),\dot\w(s_{\ell_1}^-)) \subset \Tan^\s\left(\dot\ell\big((s_{\ell_1}^-)^-\big),\dot\ell\big((s_{\ell_1}^-)^+\big)\right).
    \end{equation}
    Since $\spt(\ell)\cap\spt(\ell_1)=\{\w(s_{\ell_1}^\pm)\}$, combining \eqref{eq:tangenti_improbabili} with Lemma \ref{lem:enclosed_directions}, we deduce that $\D(\ell_1)\subset \mc E_\s(\ell)=\D(\ell)$. By Lemma \ref{lem:cut_and_paste1}\ref{c_a_p_1_item_iii} applied to $\ell_1$, with $k=\sigma\sgn(\lambda)$, we contradict the optimality of $\w$.

    \noindent\textsc{Case 2: $s_{\ell_2}^-\in(s_{\ell_1}^-,s_{\ell_1}^+)$.} \ Since $\ell_2$ is a simple, piecewise smooth loop (with one singularity), parame-trized by arc-length and with curvature of constant sign $\sigma$, we can apply Lemma \ref{lem:gauss-bonnet_nerfed}\ref{item:gb2} and deduce that $\sigma$ is the orientation of $\ell_2$, as for $\ell_1$. Since $\w(s_{\ell_1}^{\pm})$ and $\w(s_{\ell_2}^{\pm})$ lie on $\partial\D(\ell_2)$, then there exists a time $s_{\ell_1}^-<t^*<s_{\ell_2}^-$ such that $\w(t^*)\in\D(\ell_2)$. Then, we are again in position to apply Lemma \ref{lem:cut_and_paste1}\ref{c_a_p_1_item_iii} with $k=\sigma\sgn(\lambda)$, contradicting the optimality of $\w$. 

    Proof of \ref{item:lem_fortissimo_iii}. \ By Lemma \ref{lem:cut_and_paste1}\ref{c_a_p_1_item_iv} applied to the times $\tau_i,\tau_{i+1}$ (and noting that $\w_1(\tau_j)>0$, for all $j>0$, so that $\widetilde{P}(\w(\tau_j))=0$ implies $P(\w(\tau_j))=0$), it follows that $\w|_{I_i}$ must have a loop. In addition, if the first simple loop $\ell$ of $\w|_{I_i}$ is such that $\w_2(s_\ell^{-})>K\e$, then, by Lemma \ref{lem:optimal}\ref{lem:optimal_item_i}, it follows
    \begin{equation}
        \w_2(t)\geq \frac{K}{2}\e,\qquad\forall t\in[s_\ell^-,L(\w)].
    \end{equation}
    As a consequence, applying Lemma \ref{lem:cut_and_paste2}\ref{c_a_p_2_item_ii}, it follows that $\w|_{[s_\ell^+,\tau_{i+1}]}$ must be injective. Therefore, we can argue as in item \ref{item:lem_fortissimo_ii}, to conclude that the curves $\w|_{[\tau_i,s_\ell^+)}$ and $\w|_{[s_\ell^+,\tau_{i+1}]}$ cannot intersect, thus proving that $\ell$ is the unique loop of $\w|_{I_i}$.    
\end{proof}

\begin{lemma}\label{lem:sign_of_intervals}
	 For every $\e,s>0$ and $\w\in\CO$, we have:
	\begin{enumerate}[(i)]
		\item\label{item:sign_of_int_i} if $i\in\I_+$ and $i+1\in\I$, then $ i+1\in\I_-$;
		\item\label{item:sign_of_int_ii} if $\w|_{I_i}$ admits a unique loop, then $\lambda\cdot(\w_2(\tau_{i+1})-\w_2(\tau_i))\leq 0$.
	\end{enumerate} 
\end{lemma}

\begin{proof}
    Proof of \ref{item:sign_of_int_i}. By contradiction, suppose that $i,i+1\in\I_+$. Hence, the curve $\w$ must remain inside the convex set $\{\widetilde P\geq 0\}$, in a sufficiently small neighborhood of $\tau_{i+1}$ inside $I_\w$, and must be tangent to $\{P=0\}$ at $t=\tau_{i+1}$ (note that $i,i+1\in\I$ implies $\tau_{i+1}\in (0,L(\w))$ and $\w_2(\tau_{i+1})>0$, cf.\ Remark \ref{rmk:relative_position_first_times}). Thanks to Lemma \ref{lem:w1>0}, we know that $\w(\tau_{i+1})\in\{P=0,x_1>0\}$ and, by \eqref{eq:normalpolar}, the curvature of $\w$ is zero at $t=\tau_{i+1}$. However, an explicit computation shows that a curve tangent to $\{P=0\}$ with zero curvature at the tangency point must cross it. This gives a contradiction.
    
    %%%%%%%%%%%%%%%%%%%%%%%%%%%%%%%%%%%%%%
    %%%%%% The explicit computation %%%%%%
    %%%%%%%%%%%%%%%%%%%%%%%%%%%%%%%%%%%%%%
    % Indeed, recall that $P(\g_\e)\equiv 0$ and that, up to reparametrization, $\w(t_0)=\g(t_0)$, $\dot\w(t_0)=\dot\g(t_0)$ and $\ddot\w(t_0)=0$ for some $t_0$. Then, we have
    % \begin{equation}
    %     P(\w(t)) = P(\w(t))-P(\g(t)) = -\frac12 \ddot\g(t_0)^T {\rm Hess} P(\g(t_0))\ddot\g(t_0) (t-t_0)^2 +o(t-t_0)^2,
    % \end{equation}
    % as $t\to t_0$. Computing explicitly the Hessian of $P$, and using the expression of the curve $\g$ we see that the term 
    % \begin{equation}
    %     \ddot\g(t_0)^T {\rm Hess} P(\g(t_0))\ddot\g(t_0) > 0,
    % \end{equation}
    % which implies that, as $t\to t_0$, $P(\w(t))<0$, hence $\w$ has crossed the curve $\{P=0\}$.
    
    Proof of \ref{item:sign_of_int_ii}. \ Let $\ell$ be the unique loop of $\w|_{I_i}$, with corresponding interval $J_\ell=[s_\ell^-,s_\ell^+]$, and let $\s\ceq \sgn(\dot\theta|_{(\tau_i,\tau_{i+1})})$. Assume by contradiction that $\lambda\cdot(\w_2(\tau_{i+1})-\w_2(\tau_i))>0$. Then, there are four cases, according to the sign of $\lambda$ and the sign of $i$. 
    
    \noindent \textsc{Case 1: $\lambda<0$, $i\in \I_+$}. By Lemma \ref{lem_polarnormal}, we have $\s=-1$, i.e., $\w|_{I_i}$ has non-positive curvature. Additionally, $\lambda\cdot(\w_2(\tau_{i+1})-\w_2(\tau_i))>0$ implies $\w_2(\tau_{i+1})<\w_2(\tau_i)$. Consider the curve
    \begin{equation}
    \label{eq:def_eta}
        \eta\ceq \w|_{[\tau_i,s_\ell^-]}\ast \w|_{[s_\ell^+,\tau_{i+1}]}\ast \pr(\g)|_{[\w_{2}(\tau_{i+1}),\w_{2}(\tau_i)]},
    \end{equation}
    Note that $\pr(\g)$ has non-positive signed curvature, as $\w|_{I_i}$. Finally, the angles at the singularities of $\eta$ in $\w(\tau_i)$ and $\w(\tau_{i+1})$ lie in $[-\pi,0]$ by construction. Thus, we can apply Lemma \ref{lem:gauss-bonnet_nerfed}\ref{item:gb2}, with $\sigma=-1$, to deduce that $\D(\eta)=\E_{-1}(\eta)=\E(\eta)$. By Lemmas \ref{lem:enclosed_directions} and \ref{lem:fortissimo}\ref{item:lem_fortissimo_i}, we infer that $\E(\ell)\subset\E(\eta)$. This is in contradiction with Lemma \ref{lem:cut_and_paste1}\ref{c_a_p_1_item_iii}. The case $\lambda>0$ and $i\in \I_+$ is analogous.

    \noindent \textsc{Case 2:} $\lambda<0$, $i\in \I_-$. Note that, once again, $\w_2(\tau_{i+1})<\w_2(\tau_i)$. Moreover, it holds $\s=1$, i.e., the curvature of $\w|_{I_i}$ is non-negative, while the curvature of $\pr(\g)$ is non-positive. Let $\eta$ be the curve defined in \eqref{eq:def_eta} and denote by  $\delta_0\ceq \angle(\dot\w(s_\ell^-),\dot\w(s_\ell^+))$, and by $\delta_1, \delta_2$ the angles at the singularities of $\eta$ in $\w(\tau_i)$ and $\w(\tau_{i+1})$, respectively. By Lemma \ref{lem:fortissimo}\ref{item:lem_fortissimo_i}, $\delta_0\in (-\pi,0)$, while, by construction, $\delta_1,\delta_2\in [0,\pi]$. We claim that $\sigma=1$ is the orientation of $\eta$. We proceed by contradiction, assuming that $\eta$ is negatively oriented. Define the curve
    \begin{equation*}
        \tilde\eta \ceq \w|_{[\tau_i,s_\ell^-]}\ast \w|_{[s_\ell^+,\tau_{i+1}]}\ast [\w(\tau_{i+1}),\w(\tau_i)].
    \end{equation*}
    Observe that $\tilde \eta$ is obtained from $\eta$ by replacing the arc $\pr(\g)|_{[\w_2(\tau_{i+1}),\w_2(\tau_i)]}$ with the segment joining $\w(\tau_{i+1})$ with $\w(\tau_{i})$. In particular, $\tilde\eta$ is homotopic to $\eta$, it has non-negative curvature and it has three singularities at $\w(s_\ell^-)$, $\w(\tau_i)$, and $\w(\tau_{i+1})$ with angles, respectively, $\tilde\delta_0$, $\tilde\delta_1$, and $\tilde\delta_2$. By construction $\tilde\de_0=\de_0$, while the convexity of $\{P=0,x_1\geq0\}$ implies $ \tilde\delta_j\geq \delta_j-\pi/2$, for $j=1,2$. Since $\tilde\eta$ is homotopic to $\eta$, it is negatively oriented as well. Therefore, by Gauss-Bonnet Theorem (cf.\ Remark \ref{rmk:gauss-bonnet}), $\delta_0\in(-\pi,0)$, and $\de_j\in[0,\pi]$ for $j=1,2$, we deduce that
    \begin{equation*}
        2\pi = \int_{\tilde \eta} \kappa_{\tilde\eta} - \sum_{j=0}^2\tilde\delta_j \leq -\sum_{j=0}^2\delta_j +\pi < 2\pi,
    \end{equation*}
    which is a contradiction. Therefore, we deduce that $\eta$ is positively oriented. The conclusion of the proof now follows the same argument as above. The case $\lambda<0$ and $i\in\I_-$ is analogous.
\end{proof}

\begin{lemma}
	\label{lem:1stloop_is_+_or_-}
	There exists $\e_0>0$ such that for every $0<\e<\e_0$, $s>0$, and $\w\in\CO$, the first simple loop $\ell$ of $\w$ satisfies either $J_\ell\subset I_0$ or $J_\ell\subset I_1$. 
\end{lemma}

\begin{proof}
    Recall that $\w$ is not injective by Lemma \ref{lem:preliminary}\ref{lem_prelim_item_ii}. Then, if $\w|_{I_0}$ is not injective, the first simple loop of $\w$ is contained in $I_0$. If, on the contrary, $\w|_{I_0}$ is injective, since $0\in \I_+$ we must have that $1\in \I_-$ by Lemma \ref{lem:sign_of_intervals}\ref{item:sign_of_int_i}. By Lemma \ref{lem:fortissimo}\ref{item:lem_fortissimo_iii}, we deduce that $\w|_{I_1}$ admits a first simple loop, which is in fact the first simple loop of $\w$, being $\w|_{I_0}$ injective.  
\end{proof}

\section{Uniqueness of the loop}
\label{sec:looooooop}

In this section, we study the behavior of the first simple loop of optimal competitors. In particular, we prove that the first loop of an optimal competitor must be its unique loop, see Corollary \ref{cor:uniqueness_of_massa_del_sole}. This result is a consequence of the following proposition.

\begin{proposition}
	\label{prop:loops>eps/2}
	There exist $C,\e_0>0$ such that, for every $(\e,s)\in \mc I_{\e_0}$ and every $\w\in\CO$, the first simple loop $\ell$ of $\w$ satisfies
	\begin{align}
		\label{eq:loops>eps/2} 
		\w_1(t)\geq C\e^q \qquad \text{and} \qquad \w_2(t)\geq C \e, \qquad \text{for all } t\in J_\ell.
	\end{align}
\end{proposition}

The proof of Proposition \ref{prop:loops>eps/2} is given in the next section. More in details, Proposition \ref{prop:L(1stloop)>dist2} provides a lower bound on the length of the first simple loop of an optimal competitor. This bound implies $\w_2(t_\ell)\geq C\e$, with $C>0$ independent of $\e$ and $s$, cf.\ Lemma \ref{lem:geometric_w2>e2}. Then, one concludes combining Lemma \ref{lem:optimal} and the upper bound on the length of the loop given by Corollary \ref{cor:rough_bound_loops}.

\subsection{Proof of Proposition \ref{prop:loops>eps/2}}

\label{sec:proof_prop_61}

We start with the following definition. 

\begin{definition}
	\label{def:beta_loop}
    For $\e,s>0$ and $\w\in\CC$, let $\ll$ be a loop of $\w$. We set
	\begin{equation}
		\label{eq:beta_loop}
		\beta_\ll\ceq \max_{t\in J_\ell}|P(\w(t))|
		\qquad \text{and}\qquad 
		t_\ll\in\arg\max_{t\in J_\ell}|P(\w(t))|.
	\end{equation}
	Furthermore, define $(x_\ll,y_\ll)\ceq (\w_1(t_\ll),\w_2(t_\ll))$ and $\de_\ell\ceq \beta_\ell \min\{x_\ell^{-1},|y_\ell|^{-q}\}$.    
\end{definition}
 
\begin{remark}
    Before proving some preliminary results, let us report here a general upper bound for the weighted area enclosed by a loop, which follows from \eqref{eq:rado}. If $\ell$ is a loop of an optimal competitor $\w\in\CO$, for some $\e,s>0$, then we have 
\begin{align}
		\label{eq:uppbound_A(l)}
		|A(\ll)| \leq \pi^{-1}
		\beta_\ll L(\ll)^2
		\sup_{x\in \E(\ll)}x_1=	\pi^{-1}
		\beta_\ll L(\ll)^2\max_{t\in J_\ell}\w_1(t)\leq \pi^{-1}
		\beta_\ll L(\ll)^2(x_\ll+L(\ll)),
	\end{align}
where the last inequality follows from the fact that $|\dot \omega_1|\leq 1$.
\end{remark}

The next two lemmas contain some technical estimates involving the length of the first simple loop $\ell$, and the quantities defined in Definition \ref{def:beta_loop}.

\begin{lemma}
    \label{lem:newbound_L(loop)}
    There exists $\e_0>0$ such that, for all $(\e,s)\in \mc I_{\e_0}$ and all $\w\in\CO$, any loop $\ell$ of $\w$ satisfies
    \begin{equation}
        L(\ell)<2^{9}\pi^{-1}x_\ell\beta_\ell\e^{-b}\qquad \text{and}\qquad L(\ell)< x_\ell.
    \end{equation}
\end{lemma}

\begin{proof}
Consider the square $R^\e_{\alpha}\ceq [\e^q,\e^q+\alpha]\times[\e-\alpha,\e]$, for some $\alpha=o(\varepsilon^q)$, as $\e\to0$. Then, the weighted area enclosed by $R^\e_\alpha$ is
\begin{equation}
    \begin{split}
        |A(\partial R^\e_{\alpha})|&=\int_{\e-\alpha}^\e\int_{\e^q}^{\e^q+\alpha}4x_1(x_1^2-x_2^b)dx_1dx_2 \\
        &=\alpha\left((\varepsilon^q+\alpha)^{4}-\varepsilon^{2b}\right)- \frac2{b+1}\left((\varepsilon^q+\alpha)^{2}-\varepsilon^{b}\right)\left(\varepsilon^{b+1}-(\varepsilon-\alpha)^{b+1}\right)\\
        &=\alpha\varepsilon^{2b}\left(\left(1+\frac\alpha{\e^{q}}\right)^{4}-1\right)-2 \frac{\e^{2b+1}}{b+1}\left(\left(1+\frac\alpha{\varepsilon^q}\right)^{2}-1\right)\left(1-\left(1-\frac\alpha{\varepsilon}\right)^{b+1}\right)
    \end{split}
\end{equation}
Since $\alpha=o(\e^q)$, we can compute a Taylor expansion of $|A(\partial R^\e_{\alpha})|$, as $\e\to 0$. The first-order term vanishes, hence we obtain, for $\e>0$ small enough,
\begin{equation}
\label{eq:contorettangolo}
    |A(\partial R^\e_{\zz}) |=4\alpha^3 \e^{b}(1 + o(1))\geq 2\alpha^3 \e^{b}.   
\end{equation}
     Fix $\e_0>0$ such that $2^9\pi^{-1}\e_0^{M-b}<1$ (recall that $M>b$), and such that \eqref{eq:contorettangolo} holds for $\e<\e_0$, and choose $\alpha=\frac{L(\ell)}{8}$. In this way, for all $(\e,s)\in \I_{\e_0}^M$, we have $\a=o(\e^q)$ by Corollaries \ref{cor:newLS} and \ref{cor:rough_bound_loops}. Applying Lemma \ref{lem:cut_and_paste1}\ref{c_a_p_1_item_i} to the curve that bounds the square $R^\e_\zz$, we have that $|A(\partial R^\e_\alpha)|\leq |A(\ell)|$. Thus, combining the latter inequality with \eqref{eq:uppbound_A(l)} and \eqref{eq:contorettangolo}, we obtain
    \begin{equation}
        \label{eq:best_upboubd_L(loop)}
        L(\ell)\leq 2^8\pi^{-1}\beta_\ell \e^{-b}(x_\ell+L(\ell)).
    \end{equation}
    Since $\w\in \CO$, using Corollary \ref{cor:newLS}, we have
    $2^9\pi^{-1}\beta_\ell\e^{-b}<1$. Then, if $x_\ell\leq L(\ell)$, by \eqref{eq:best_upboubd_L(loop)}, we immediately get a contradiction. Therefore, $L(\ell)<x_\ell$, and we get the claim from \eqref{eq:best_upboubd_L(loop)}.   
\end{proof}

\begin{lemma}
    \label{lem:newlambdabeta}
    There exists $\e_0>0$ such that, for all $(\e,s)\in \mc I_{\e_0}^M$ and $\w\in\CO$, the first simple loop $\ell$ of $\w$ satisfies
    \begin{equation}
        |\la|(\beta_\ell x_\ell)^{2}\geq \frac{\pi^2}{2^{12}}\e^{b}.
    \end{equation}
\end{lemma}

\begin{proof}
    Let $\e_0>0$ given by Lemma \ref{lem:newbound_L(loop)}, $(\e,s)\in \mc I_{\e_0}^M$ and $\w\in\CO$. By Lemma \ref{lem:fortissimo}\ref{item:lem_fortissimo_i}, the curvature $\k_\ell$ of the first simple loop $\ell$ of $\w$ has constant sign $\sigma\in \{-1,1\}$, and $\sigma\cdot\angle(\dot\omega(s_\ell^+),\dot\omega(s_\ell^-))\in (0,\pi)$. Thus, by Lemma \ref{lem:convexcurvature}, there exists $t_*\in \mathrm{int}(J_\ell)=(s_\ell^-,s_\ell^+)$ such that
    \begin{equation}
    \label{eq:chatgpt}
        |\kappa_\ell(t_*)|\geq \frac{\pi}{L(\ell)}.
    \end{equation}
    On the other hand, using the second inequality of Lemma \ref{lem:newbound_L(loop)}, we can estimate
    \begin{equation}
    \label{eq:skrinking}
        |\kappa_\ell(t_*)|=|4\la \w_1(t_*)P(\w(t_*))|\leq 4|\la|(x_\ell+L(\ell))\beta_\ell \leq 8|\la| x_\ell\beta_\ell.
    \end{equation}
    Combining \eqref{eq:chatgpt} and \eqref{eq:skrinking}, we get
    \begin{equation}
        |\la|L(\ell)x_\ell\beta_\ell\geq \frac{\pi}{8}.
    \end{equation}
    Thus, using the first inequality of Lemma \ref{lem:newbound_L(loop)} to estimate $L(\ell)$ from above, we obtain the result.
\end{proof}

We next show that, if the length of a loop of $\w$ is lower bounded by $\delta_\ell$, where $\de_\ell$ is defined in Definition \ref{def:beta_loop}, then $\w_2(t_\ell)$ lies in the region $\{x\in\R^2\mid x_2\geq C\e\}$, for a suitable $C>0$.

\begin{lemma}
    \label{lem:geometric_w2>e2}
	For every $K>0$ there exist $\e_0,K'>0$ such that, for every $(\e,s)\in\I_{\e_0}$ and any loop $\ell$ of $\w\in\CO$ we have 
	\begin{equation}
		\label{eq:L(loop)>dist2}
		L(\ll)\geq K\de_\ell \qquad\implies\qquad y_\ll \geq K'\e.
	\end{equation}
\end{lemma}

\begin{proof}
Proceeding with hindsight, define $\bar K\ceq 2^{-9}\pi K$. Fix $\e_0>0$ smaller than the one given by Lemma \ref{lem:newbound_L(loop)} and small enough such that $\e_0<\frac12 \bar K^2$ and, for every $(\e,s)\in \mc I_{\e_0}$,
    \begin{align}
        \label{eq:Lgse}L(\gse)= s+\e+\frac{q^2}{2(b-1)}\e^{b-1}+o(\e^{b-1})\leq  \Big(1+\frac{\bar K^{-q}}{2}\Big)\e, 
    \end{align}
where we used the fact that $s\leq \e^2$. Now, fix $(\e,s)\in\I_{\e_0}$ and $\w\in\CO$ with first simple loop $\ell$. On the one hand, by \eqref{eq:uppbound_A(l)} and the second inequality of Lemma \ref{lem:newbound_L(loop)}, we have
\begin{equation}
\label{eq:uppbound_A(l)_updated}
    |A(\ell)|\leq \pi^{-1}
		\beta_\ll L(\ll)^2(x_\ll+L(\ll)) \leq 2\pi^{-1}
		\beta_\ll x_\ll L(\ll)^2.
\end{equation}
On the other hand, consider the square $R^\e_\zz$ as in the proof of Lemma \ref{lem:newbound_L(loop)} with $\zz=\frac{L(\ell)}{8}$, so that $L(\partial R^\e_\zz)<L(\ell)$. Hence, combining \eqref{eq:contorettangolo}, with the lower bound on $L(\ell)$ of \eqref{eq:L(loop)>dist2}, and with Lemma \ref{lem:cut_and_paste1}\ref{c_a_p_1_item_i}, we obtain
\begin{align}
    \label{eq:lowbound_A(R)}
    |A(\ell)|\geq|A(\partial R^\e_\a)|\geq 2^{-8} L(\ell)^3 \e^b \geq 2^{-8} K \delta_\ell L(\ell)^2 \e^b  = 2^{-8} K \beta_\ell L(\ell)^2\e^{b}\min\{x_\ell^{-1},|y_\ell|^{-q}\}.
\end{align}
Therefore, combining \eqref{eq:uppbound_A(l)_updated} and \eqref{eq:lowbound_A(R)}, we get
	\begin{equation}
		\label{eq:xl+L(l).max>eps^{2(b-q)}}
		x_\ll\max\{x_\ll,|y_\ll|^{q}\}\geq \bar K\e^{b}.
	\end{equation}
If $x_\ll\geq|y_\ll|^{q}$, then $\beta_\ell=x_\ell^2-y_\ell^b$ and \eqref{eq:xl+L(l).max>eps^{2(b-q)}} implies $x_\ll\geq \bar K \e ^{q}$. By Corollary \ref{cor:newLS} we also have $\beta_\ll<\e^M$, with $M>b$, and thus, by our choice of $\e_0$, we get
	\begin{equation}
    \label{eq:lem64_claim1}
	y_\ll^b=
	x_\ll^2-\beta_\ll\ge 
	\bar K^2\e^b - \e^M \geq \bar K^2\e^b\Big(1-\frac{\e^{M-b}}{\bar K^2}\Big)\geq \frac{\bar K^2}{2}\e^b,
	\end{equation}
which proves the claim of \eqref{eq:L(loop)>dist2}. If instead $x_\ll\leq|y_\ll|^{q}$, then \eqref{eq:xl+L(l).max>eps^{2(b-q)}} implies $|y_\ll|^{b}\geq \bar K\e^{b}$. We claim that $y_\ell>0$, which in turn implies that
\begin{equation}
    \label{eq:lem64_claim2}
    y_\ll^{b}\geq \bar K\e^{b}.
\end{equation}
Assume by contradiction that $y_\ell\leq 0$, and thus $y_\ell\leq -\bar K^{-b}\e$. This implies that the optimal competitor $\w\in\CO$ satisfies
\begin{equation}
    L(\gamma_{s,\e})\geq L(\w)\geq L([(x_\ell,y_\ell),(\e^q,\e)])\geq |y_\ell-\e|\geq(1+\bar K^{-b})\e,
\end{equation}
which is in contradiction with \eqref{eq:Lgse}. Hence, \eqref{eq:lem64_claim2} holds, and this, together with \eqref{eq:lem64_claim1} completes the proof of the lemma with $K'=\max\{\bar K^{-b},\frac{\bar K^2}{2}\}$.
\end{proof}	

We next provide two lower bounds that we need to prove Proposition \ref{prop:L(1stloop)>dist2} below: we show that $x_\ell$ can be lower bounded with $\sqrt{\beta_\ell}$, see Lemma \ref{lem:svolta}, and that $|\lambda|\beta_\ell^2$ is lower bounded by a positive constant, see Lemma \ref{lem:new_lb2_smaller1}.

\begin{lemma}
    \label{lem:svolta}
    There exist $\e_0>0$ such that for all $(\e,s)\in \I_{\e_0}$, the first simple loop $\ell$ of every optimal competitor $\w\in\CO$ satisfies
    \begin{equation}
        x_\ell\geq \frac12\sqrt{\beta_\ell}.
    \end{equation}
\end{lemma}

\begin{proof}
    Let $\e_0>0$ be given by Lemma \ref{lem:geometric_w2>e2}. Take $(\e,s)\in \I_{\e_0}$ and $\w\in\CO$ with first simple loop $\ell$. We distinguish the two cases where $\ell$ is either positive or negative. 

    \noindent \textsc{Case 1:} $\ell$ \textsc{negative}. The loop $\ell$ must be contained in the region $\{P<0\}\cap\{x_2>0\}$, hence relying on Lemma \ref{lem:geometric_w2>e2}, we can repeat verbatim the proof of \cite[Lem.\ 3.11]{notall}. 
    
    \noindent \textsc{Case 2:} $\ell$ \textsc{positive}. Assume by contradiction that $x_\ell<\frac12\sqrt{\beta_\ell}$. Then, since $x_\ell^2-y_\ell^b=\beta_\ell$ (recall that $\ell$ is positive) and $b$ is odd, we must have $y_\ell<0$. Hence, we deduce that
    \begin{equation}
        |y_\ell|^b=\beta_\ell-x_\ell^2 > \frac34\beta_\ell.
    \end{equation}
    From the latter inequality and $x_\ell<\frac12\sqrt{\beta_\ell}$ we have
    \begin{equation}
        \label{eq:xell<<yell}
        x_\ell<\frac12\beta_\ell^{(\frac12-\frac1b)}\beta_\ell^{\frac1b}<\frac12\Big(\frac43\Big)^{\frac1b}\beta_\ell^{(\frac12-\frac1b)}|y_\ell|<|y_\ell|,
    \end{equation}
    having used that $\beta_\ell<1$. Since $y_\ell<0$ and $\w$ is parametrized by arc-length, \eqref{eq:xell<<yell} implies that
    \begin{equation}
        \w_2([t_\ell-x_\ell,t_\ell+x_\ell])\subset\{x_2<0\},
    \end{equation}
    and thus, since $\w_1>0$ by Lemma \ref{lem:w1>0}, we also have $(P\circ\w)|_{[t_\ell-x_\ell,t_\ell+x_\ell]}>0$. Therefore, $\w|_{[t_\ell-x_\ell,t_\ell+x_\ell]}\subset I_0$ and, since $L(\ell)<x_\ell$ by Lemma \ref{lem:newbound_L(loop)}, $\w|_{[t_\ell-x_\ell,t_\ell+x_\ell]}$ contains the loop $\ell$. By Lemma \ref{lem:fortissimo}\ref{item:lem_fortissimo_ii}, $\ell$ must be the unique loop of $\w|_{[t_\ell-x_\ell,t_\ell+x_\ell]}$. Consider now the interval $I\ceq [t_\ell,t_\ell+\frac{x_\ell}{2}]$ and the integral
    \begin{equation}
        J\ceq \int_I |\kappa(t)|dt= 4\int_I |\la|\w_1(t)P(\w(t))dt.
    \end{equation}
    We aim to reach a contradiction by estimating from above and from below the integral $J$. On the one hand, by Lemma \ref{lem:fortissimo}\ref{item:lem_fortissimo_i} and the Gauss-Bonnet theorem, we have $J\leq 6\pi$. On the other hand, the unit speed parametrization of $\w$ implies that 
    \begin{equation}
        \w(I)\subset \Big[\frac{x_\ell}{2},\frac{3x_\ell}{2}\Big]\times\Big[y_\ell-\frac{x_\ell}{2},y_\ell+\frac{x_\ell}{2}\Big].
    \end{equation}
    Thus, for every $t\in I$, we have $\w_1(t)\geq \frac{x_\ell}{2}$ and
    \begin{equation}
    \label{eq:lower_bound_P_|I}
        P(\w(t))\geq P\Big(\frac{x_\ell}{2},y_\ell+\frac{x_\ell}{2}\Big)=\frac{x_\ell^2}{4}-y_\ell^b\Big(1+\frac{x_\ell}{2y_\ell}\Big)^b\geq \frac{x_\ell^2}{4}-\frac{y_\ell^b}{2^b}\geq \frac{\beta_\ell}{2^b},
    \end{equation}
    where in the second-to-last inequality, we used \eqref{eq:xell<<yell} (recall that $b$ is odd and $y_\ell<0$). Then, using \eqref{eq:lower_bound_P_|I} and Lemma \ref{lem:newlambdabeta} for estimating from below $|\lambda|$, we obtain
    \begin{equation}
        J\geq |\la|\frac{x_\ell^{2}\beta_\ell}{2^{b}}\geq\frac{\pi^2}{2^{b+12}}\frac{\e^{b}}{\beta_\ell}.
    \end{equation}
    This is in contradiction with $J\leq6\pi$, since, as $\e\to 0$, $\beta_\ell=o(\e^M)$ with $M>b$, by Corollary \ref{cor:newLS}. 
\end{proof}

\begin{lemma}
	\label{lem:new_lb2_smaller1}
	There exist $C,\e_0>0$ such that, for all $(\e,s)\in \I_{\e_0}$ and $\w\in\CO$, the first simple loop $\ell$ of $\w$ satisfies
	\begin{equation}
		\label{eq:new_lb2_smaller1}
		|\la| \beta_\ll^2\leq C.
	\end{equation} 
\end{lemma}

\begin{proof} 
    Let $\e_0>0$ be given by Lemma \ref{lem:geometric_w2>e2}. Take $(\e,s)\in \I_{\e_0}$ and $\w\in\CO$ with first simple loop $\ell$. Set $N\ceq  \max \{|y_\ll|^q , \sqrt{\beta_\ll}\}$. We claim that  
	\begin{equation}
		\label{eq:bound_xbeta}
		\frac {N}{4}\leq x_\ll  \le 4 N.
	\end{equation}
    To prove \eqref{eq:bound_xbeta}, we distinguish two cases, according to the sign of $P(x_\ll,y_\ll)$.

    \noindent \textsc{Case 1: $P(x_\ell,y_\ell)>0$.} We have $0<x_\ll^2 = y_\ll^b+\beta_\ll$. If $N=|y_\ll|^q$, then $|y_\ell|^b\geq \beta_\ell$ and, in particular, $y_\ll^b> 0$. Hence, we deduce that
	\[
	N^2= y_\ll^b \leq y_\ll^b +\beta_\ll =      x_\ll^2\leq 2 y_\ll^b = 2N^2,
	\]
    which implies $N\leq x_\ell\leq \sqrt{2}N$.
	If instead $N= \sqrt{\beta_\ll}$, then, by Lemma \ref{lem:svolta}, we have $x_\ll \geq  \frac12\sqrt{\beta_\ll}=\frac12 N$. While the upper bound is given by $x_\ll^2 = y_\ll^b+\beta_\ll\leq  |y_\ll^b|+\beta_\ll\leq 2 \beta_\ll=2N^2$. Thus, $2^{-1}N\leq x_\ell\leq \sqrt 2 N$.

	\noindent \textsc{Case 2: $P(x_\ell,y_\ell)\leq0$.} We have $0<x_\ll^2 = y_\ll^b-\beta_\ll$, which implies that $0\leq \beta_\ll<y_\ll^b$ and thus $ N= y_\ll^q$. Therefore, we easily see that $x_\ll = (y_\ll^b-\beta_\ll)^{\frac12}\leq y_\ll^q=N$. For the lower bound, we distinguish two cases: either $2 \beta_\ll<y_\ll^b$ or $2\beta_\ll\geq y_\ll^b$. In the first case, we have	$x_\ll^2 = y_\ll^b-\beta_\ll\geq \frac 12 y_\ll^b =\frac 12 N^2$, or equivalently, $x_\ell\geq 2^{-2}N$. In the other case, we use Lemma \ref{lem:svolta} to obtain
	\[
	 x_\ll \geq \frac1{2}  \sqrt{\beta_\ll} \geq  \frac1{2\sqrt 2}y_\ell^q = \frac1{2\sqrt 2} N.
	\] 
    Claim  \eqref{eq:bound_xbeta} is proved. A direct consequence of it is that  \begin{equation}\label{eq:betal-x1}
	\frac{\beta_\ll}{4N}\le  \frac{\beta_\ll}{x_\ll}	\le    4 N.
    \end{equation}
    Consider the interval $\dsy I_c\ceq \Big[t_\ll-c\frac{\beta_\ll}{x_\ll},t_\ll\Big]$, where $c\in \left(0,\frac14\right]$ has to be chosen later. Observe that, by Lemma \ref{lem:svolta}, and since $\w$ is parametrized by arc-length, we have
    \begin{equation}
        t_\ell\geq x_\ell  \geq \frac{\beta_\ell}{4x_\ell} \geq c\frac{\beta_\ll}{x_\ll},    
    \end{equation}
    and hence $I_c\subset [0,L(\w)]$. By $|\dot\w_1|\leq1$, \eqref{eq:bound_xbeta} and \eqref{eq:betal-x1}, for $t\in I_c$ we have
    \begin{equation}
    \label{eq:bound_omega1beta}
	\w_1(t)\leq x_\ell+|\w_1(t) - x_\ll| \leq x_\ell+c\frac{\beta_\ll}{x_\ll} \leq 4(1+c)N,
    \end{equation}
    and, similarly,
    \begin{equation}
    \label{eq:bound_omega1beta_below}
	\w_1(t)\geq x_\ell-|\w_1(t) - x_\ll| \geq x_\ell-c\frac{\beta_\ll}{x_\ll} \geq \left(\frac{1}{4}-4c\right)N.
    \end{equation}
    Furthermore, using $|\dot\w_2|\leq 1$ and \eqref{eq:betal-x1}, and since $0<N<1$ is small, we also have for $t\in I_c$
    \begin{equation}
	\label{eq:bound_omega2beta}
	|\w_2(t)| \leq |y_\ll| + |\w_2(t)-y_\ell| \leq |y_\ll| +  c\frac{\beta_\ll}{x_\ll} \leq N^{\frac{1}{q}} +  4c N \leq (1+ 4c) N^{\frac{1}{q}}.
    \end{equation}
    Using \eqref{eq:bound_omega1beta} and \eqref{eq:bound_omega2beta}, on $I_c$ we can bound the derivative of $P(t)=P(\omega(t))$, for every $t\in I_c$:
    \begin{equation}
        |\dot P(t)| \leq 2\w_1(t)+b |\w_2(t)|^{b-1}\leq 8(1+c) N + b(1+4c)^{b-1} N^{2-\frac{1}q} \leq C_1 N, \label{dotP}
    \end{equation}
	where $C_1$ is a constant only depending on $b$.	We now fix $c\in \left(0,\frac14\right]$ such that 
	\[
    \frac1{4}-4c\ge \frac1{8}\qquad \mbox{and}\qquad  4cC_1  \leq \frac12.
    \]
	From
	\eqref{dotP} and \eqref{eq:betal-x1}, we deduce that, for every $t\in I_c$,
	\begin{equation}
		\label{boundP}
		|P(t)| \geq \beta_\ll - c C_1\frac{\beta_\ll}{x_\ll} N
		\geq \beta_\ll - 4cC_1 \beta_\ll
		\geq \frac{\beta_\ll}2.
	\end{equation}
	In particular, the sign of $P(t)$ does not change for every $t\in I_c$, and thus we can apply Lemma \ref{lem:fortissimo}\ref{item:lem_fortissimo_i} to bound the total curvature of $\w|_{I_c}$: 
    \begin{equation}
        \label{eq:bound_total_curvature_aux}
        \int_{I_c}|\kappa(t)|dt=\int_{I_c}|\dot\theta(t)|dt\leq 4\pi, 
    \end{equation}
    where $\theta$ is the angle map of $\w$. Then, using the expression \eqref{eq:normalpolar} for the angle map, and combining \eqref{eq:bound_xbeta}, \eqref{eq:bound_omega1beta_below}, \eqref{boundP} and \eqref{eq:bound_total_curvature_aux}, we conclude that
    \begin{equation}
    \begin{split}
        4\pi\geq \int_{I_c} |\dot\0(t)| dt  &= \int_{I_c} 4|\lambda| \w_1(t) |P(t)| dt \geq        2|\lambda|c\frac{\beta_\ll^2}{x_\ll} \left(\frac14-4c\right)N \geq \frac1{16}|\lambda|c\beta_\ll^2, 
    \end{split}
	\end{equation} 
    which proves the lemma with $C\ceq 2^6 c^{-1}\pi$.
\end{proof} 

We are now ready to show that the length of the first simple loop is lower bounded by $\delta_\ell$, up to the multiplication by a positive constant.

\begin{proposition}
	\label{prop:L(1stloop)>dist2}
	There exist $C,\e_0>0$ such that, for all $(\e,s)\in \I_{\e_0}$ and all $\w\in\CO$, the first simple loop $\ell$ of $\w$ satisfies $L(\ll)\geq C\de_\ell$.
\end{proposition}

\begin{proof}
	Let $\e_0>0$ be given by Lemma \ref{lem:new_lb2_smaller1}. Take $(\e,s)\in \I_{\e_0}$ and $\w\in\CO$ with first simple loop $\ell$. Let $J_\ell$ the interval of $\ell$ and recall that $L(\ell)=|J_\ell|$. By Lemma \ref{lem:newbound_L(loop)}, $L(\ll)\leq x_\ll$, while by Lemma \ref{lem:fortissimo}\ref{item:lem_fortissimo_i}, the total curvature of $\ell$ is at least $\pi$. Therefore, using \eqref{eq:normalpolar} for the angle map, we obtain
	\[
	\begin{split}
	\pi & \le  \int_{J_\ell}|\dot\theta(t)|dt 
	= |\lambda|
	\int_{J_\ell}  |Q(\omega(t))|dt\le |\lambda| |J_\ell| \max_{t\in J_\ell}|Q (\omega(t))|
	\\
	& \leq 4 |\lambda| L(\ll)\beta_\ll   \max_{t\in J_\ell} \w_1(t)\leq  4 |\lambda| L(\ll)\beta_\ll  (x_\ll +L(\ll)) \\
    &\leq 8|\la|L(\ell)\beta_\ell x_\ell,
	\end{split}
	\]
	and thus we have
    \begin{equation}
        L(\ell)\geq \frac{\pi}{8|\la|\beta_\ell x_\ell} = \frac{\pi}{8}\frac{1}{|\la|\beta_\ell^2}\frac{\beta_\ell}{x_\ell} \geq \frac{\pi}{8}\frac{1}{|\la|\beta_\ell^2} \de_\ell\geq C\de_\ell.
    \end{equation}
    where in the last inequality, we have used Lemma~\ref{lem:new_lb2_smaller1}.	 
\end{proof}

We can finally complete the proof of Proposition \ref{prop:loops>eps/2}.

\begin{proof}[Proof of Proposition \ref{prop:loops>eps/2}]
Let $C,\e_0$ be the constants given by Proposition \ref{prop:L(1stloop)>dist2}. Take $(\e,s)\in \I_{\e_0}$ and $\w\in\CO$ with first simple loop $\ell$. By Lemma \ref{lem:geometric_w2>e2}, there exists $K'>0$ such that $y_\ell=\w_2(t_\ell)\geq K'\e$. Then, since $|\dot \w_2|\leq 1$, by Corollaries \ref{cor:newLS} and \ref{cor:rough_bound_loops}, it follows that
\begin{equation}
\label{eq:lower_bound_w_2_loop}
    \w_2(t)\geq \w_2(t_\ell) - L(\ell) \geq K'\e - K\e^M \geq \frac{K'}2 \e,\qquad \forall\, t\in J_\ell,
\end{equation}
up to shrinking $\e_0>0$, if necessary. The claim for $\w_1(t)$, $t\in J_\ell$, follows from the estimate
\begin{equation}
	\w_1(t)^2 = \w_2(t)^b + P(\w(t)) \geq \w_2(t)^b - \beta_\ell \geq \frac{K'}{4}\e^b,
\end{equation}
where the last inequality is a consequence of \eqref{eq:lower_bound_w_2_loop}, together with Corollary \ref{cor:newLS}.
\end{proof}

\subsection{Consequences of Proposition \ref{prop:loops>eps/2}}

\label{sec:consequence_prop_61}

We collect here some consequences of Proposition \ref{prop:loops>eps/2}, showing, in particular, that any optimal competitor admits a unique loop. 

\begin{proposition} 
	\label{prop:lambda>0}
	There exists $\e_0>0$ such that, for all $(\e,s)\in \I_{\e_0}$ and all $\w\in\CO$, $\lambda_\w<0$.
\end{proposition}

\begin{proof}
    Let $\e_0$ be as in Proposition \ref{prop:loops>eps/2}, and fix an optimal competitor $\w \in \CO$, with $(\e,s)\in \I_{\e_0}$. Assume by contradiction that $\lambda=\lambda(\w)>0$. If $\w|_{I_0}$ admits a first simple loop $\ell$, then by Lemma \ref{lem:fortissimo}\ref{item:lem_fortissimo_ii} $\ell$ is the unique loop of $\w|_{I_0}$. Thus, by Lemma \ref{lem:sign_of_intervals}\ref{item:sign_of_int_ii}, $\w_2(\tau_1)\leq \w_2(\tau_0)$, which is impossible by Remark \ref{rmk:relative_position_first_times}. Therefore, $\w|_{I_0}$ is injective and the open set $\Omega$ enclosed by $\w|_{I_0}$ and the arc of $\{\widetilde P=0\}$ from $\w(\tau_0)$ to $\w(\tau_1)$ is convex. Hence, the maximum of $\widetilde P$ in $\Omega$ is attained at a boundary point and 
    \begin{equation}
        \max_{(x,y)\in \overline\Omega} \widetilde P(x,y) = \max_{\tau \in I_0} \widetilde P(\w(\tau)) \leq \widetilde \beta \leq \e^{M},  
    \end{equation}
    where, in the last inequality, we used Corollary \ref{cor:newLS}. On the other hand, by convexity of $\Omega$, we have $\frac12\w(\tau_1)\in\Omega$. Since $\w_1(\tau_1)^2=\w_2(\tau_1)^b$, we deduce that
    \begin{equation}
        \label{eq:massa_del_sole}
        \e^M \geq \max_{(x,y)\in \overline\Omega} \widetilde P(x,y) \geq P\left(\frac12\w(\tau_1)\right) = \left(\frac14-\frac1{2^b}\right)\w_2(\tau_1)^b.
    \end{equation}    
    By Lemma \ref{lem:sign_of_intervals}\ref{item:sign_of_int_ii}, we have $\w_2(\tau_2)\leq \w_2(\tau_1)$. Moreover, as $\w|_{I_0}$ is injective, the first simple loop $\ell$ of $\w$ is negative and contained in $I_1$, as a consequence of Lemma \ref{lem:1stloop_is_+_or_-}. Thus, by Proposition \ref{prop:loops>eps/2} and Lemma \ref{lem:optimal}, we deduce that, for $t\in J_\ell$, $\w_2(\tau_2)\geq \w_2(t) \geq C\e$. Since $M>b$, this is in contradiction with \eqref{eq:massa_del_sole}, up to shrinking $\e_0$.
\end{proof}

\begin{corollary}[Uniqueness of the loop]
\label{cor:uniqueness_of_massa_del_sole}
    There exists $\e_0>0$ such that for all $(\e,s)\in \I_{\e_0}$ and all $\w\in\CO$, $\w$ has a unique loop $\ell$. Moreover, only one of the following situations may occur:
    \begin{enumerate}[(i)]
        \item\label{cor:uniqueness_of_massa_del_sole_item_i} $J_\ell\subset I_0$ and $\I=\I_+=\{0\}$, i.e., $I_0=I_\w$;
        \item\label{cor:uniqueness_of_massa_del_sole_item_ii} $J_\ell\subset I_1$ and $\I=\{0,1\}$, i.e., $I_\w=I_0\cup I_1$, with $0\in\I_+$ and $1\in\I_-$;
        \item\label{cor:uniqueness_of_massa_del_sole_item_iii} $J_\ell\subset I_1$ and $\I=\{0,1,2\}$, i.e., $I_\w=I_0\cup I_1\cup I_2$, with $0,2\in\I_+$ and $1\in\I_-$.
    \end{enumerate}
\end{corollary}

\begin{proof}
   Let $\e_0>0$ be as in Proposition \ref{prop:loops>eps/2}, and fix an optimal competitor $\w \in \CO$, with $(\e,s)\in \I_{\e_0}$. Then, the first simple loop $\ell$ of $\w$ is contained in the region $\{x_2\geq C\e\}$, for some constant $C=C(b)>0$. We distinguish two cases.
    
    \noindent\textsc{Case 1: $J_\ell\subset I_0$.} We claim that, in this case, $\ell$ is the unique loop of $\w$ and $I_0=I_\w$, so that we are in situation \ref{cor:uniqueness_of_massa_del_sole_item_i}. First, by Lemma \ref{lem:fortissimo}\ref{item:lem_fortissimo_ii}, we have that $\w|_{I_0}$ contains only the loop $\ell$. If by contradiction $I_0\neq I_\omega$, then $1\in\mc I_-\neq \emptyset$. By Lemma \ref{lem:fortissimo}\ref{item:lem_fortissimo_iii} we have that $\w|_{I_1}$ admits a loop $\bar\ell$, and by Lemma \ref{lem:optimal} $\w_2(s_{\bar\ell}^-)\geq \frac{C}{2}\e$. Thus, we obtain a contradiction by Lemma \ref{lem:cut_and_paste2}\ref{c_a_p_2_item_ii}, verifying the claim.
    
    \noindent\textsc{Case 2:} $\w|_{I_0}$ is injective. By Lemma \ref{lem:1stloop_is_+_or_-}, $J_\ell\subset I_1$. Since $\ell$ is contained in the region $\{x_2\geq C\e\}$, $\ell$ is the unique loop of $I_1$ by Lemma \ref{lem:fortissimo}\ref{item:lem_fortissimo_iii}. Moreover, $\w|_{I_i}$ must be injective for every $i\geq 2$ by Lemma \ref{lem:cut_and_paste2}\ref{c_a_p_2_item_ii}, which can be applied thanks to Lemma \ref{lem:optimal}. This implies that $\ell$ is the unique loop. Finally, Lemmas \ref{lem:fortissimo}\ref{item:lem_fortissimo_iii} and \ref{lem:sign_of_intervals}\ref{item:sign_of_int_i} imply $\mc I\subset\{0,1,2\}$, and thus we fall either in \ref{cor:uniqueness_of_massa_del_sole_item_ii} or \ref{cor:uniqueness_of_massa_del_sole_item_iii}.    
\end{proof}

For the next result, we introduce the following notation: for $\e,s>0$ and $\w\in\CO$, define
\begin{equation}
   \label{eq:timeT}
   T_0=T_0(\w)\ceq  \min\{t\in I_0 \mid \theta(t)\leq0 \}
\end{equation}
Note that the definition is well-posed since $\0$ is monotone decreasing in $I_0$, by Lemma \ref{lem_polarnormal} and Proposition \ref{prop:lambda>0}. Moreover, since $\dot\w_1=\cos\0$, $\w_1$ is monotone increasing on $[0,T_0]$.

\begin{corollary}
    \label{cor:unique_intersection_x1axis}
    There exists $\e_0>0$ such that for all $(\e,s)\in \I_{\e_0}$ and for all $\w \in \CO$, $T_0(\w)>0$ and there is a unique $t_\w\in(0,T_0(\w))$ such that $\w_2(t_\w)=0$. Moreover, $\theta(t_\w)\in(0,\frac\pi2)$.
\end{corollary}

\begin{proof}
    Let $\e_0>0$ be such that Proposition \ref{prop:loops>eps/2} and Corollary \ref{cor:uniqueness_of_massa_del_sole} holds, and fix an optimal competitor $\w \in \CO$, with $(\e,s)\in \I$. Let $\ell$ be the unique loop of $\w$, with loop interval $J_\ell$. We start by proving the following claim: 
    \begin{equation}
        \label{eq:claim_T_0}
        T_0<\tau_1 \qquad\text{if and only if} \qquad J_\ell\subset I_0.
    \end{equation}
    Assume that $J_\ell\subset I_0$ and assume by contradiction that $T_0=\tau_1$. By the definition of $T_0$, $\theta(t)>0$ for all $t\in [0,\tau_1)$. Since $\theta(0)<\frac\pi2$ by Lemma \ref{lem:w1>0} and recalling that $\theta$ is decreasing, $\theta(t)\in \left(0,\frac\pi2\right)$ for every $t\in [0,\tau_1)$, which implies $\dot\w_1(t)=\cos\0(t)>0$, for all $t\in [0,\tau_1)$. Therefore, $\w_1$ is strictly monotone in $[0,\tau_1)$ and thus, $\w|_{I_0}$ cannot have self-intersections. This is in contradiction with $J_\ell\subset I_0$. Conversely, assume that $T_0<\tau_1$. Note that either $T_0=0$ or $\theta(T_0)=0$. Consider the half-line $r(t)\ceq \w(T_0)+t\dot\w(T_0)$, $t\in[0,+\oo)$ and the curve $\G:\R\to\R^2$ defined by
    \begin{equation}
    \Gamma(t)\ceq 
    \begin{cases}
        (0,t-s), \quad & t\in (-\oo,0),\\
        \w(t), \quad & t\in[0,T_0],\\
        r(t-T_0), \quad &t\in[T_0,+\oo).
    \end{cases}
    \end{equation}
    By construction, $\R^2\setminus\spt(\Gamma)$ has two unbounded connected components and we let $\Omega$ be the one not containing $(\e^q,\e)$. Observe that, since $\w_1$ is strictly increasing on $[0,T_0]$ and $[0,T_0]\subsetneq I_0$, $\spt\left(\pr(\gse)\right)\cap \spt(\Gamma)=\{(0,-s)\}$ and $\wt P|_{\Omega}>0$. Since $\theta$ is monotone decreasing in $I_0$, it follows that $\w(T_0+\de)\in\Omega$, for small $\de>0$. Consequently, there is $t\in (T_0,\tau_1]$ such that $\w(t)\in\partial\Omega$, and we let $t_*\ceq \min\{t\in (T_0,\tau_1]\mid \ \w(t)\in\partial\Omega\}$. By Lemma \ref{lem:w1>0}, either $\w(t_*)\in\spt(\w|_{[0,T_0]})$ or $\w(t_*)\in\spt(r)$. If $\w(t_*)\in\spt(\w|_{[0,T_0]})$, $\w|_{I_0}$ has a loop $\ell$, with $s_\ell^+=t_*$. If, instead $\w(t_*)\in\spt(r)$, by \cite[Lem.\ 3.4]{notall}, $\w|_{[0,t_*]}$ has a loop $\ell$ with $s_\ell^+\leq t_*$. In both cases, $J_\ell\subset I_0$ and the claim \eqref{eq:claim_T_0} is proved. We now conclude the proof: we distinguish the two cases, where either $J_\ell\subset I_0$ or $J_\ell\subset I_1$, cf.\ Lemma \ref{lem:1stloop_is_+_or_-}.
    
    \noindent\textsc{Case 1: $J_\ell\subset I_0$}. From the proof of the claim \eqref{eq:claim_T_0} it follows that $\spt(\ell)\subset\overline{\Omega}\subset \{x_2\leq \w_2(T_0)\}$ and, in particular, $\w_2(T_0)\geq \w_2(t)\geq C\e$, for every  $t\in J_\ell$, where $C>0$ is given by Proposition \ref{prop:loops>eps/2}. Since $\w_2(T_0)\geq C\e$, then $T_0>0$ (indeed, if $T_0=0$, then $\w_2(T_0)=-s$) and, by Lemma \ref{lem:optimal}, $\w_2(t)\geq \frac{C}2\e>0$, for every $t\in [T_0,L(\w)]$. Therefore, recalling that $\w_1|_{[0,T_0]}$ is strictly increasing, $\w$ intersects exactly once the positive $x_2$-axis at a time $t_\w\in (0,T_0)$.

    \noindent\textsc{Case 2: $J_\ell\subset I_1$}. By \eqref{eq:claim_T_0}, $T_0=\tau_1>0$ and, since $\w_1|_{[0,T_0]}$ is strictly increasing, then $\w|_{[0,T_0]}$ has exactly one intersection with the $x_2$-axis. Then, on $I_1$, $\w$ cannot intersect the positive $x_2$-axis, since $\spt(\w|_{I_1})=\spt(\w|_{[T_0,\tau_2]})$ is contained in the region $\{P<0\}$. Finally, as $J_\ell\subset I_1$, combining Proposition \ref{prop:loops>eps/2} with Lemma \ref{lem:optimal}, we further deduce that $\w|_{[\tau_2,L(\w)]}$ cannot intersect the $x_2$-axis.
\end{proof}

\section{Estimates of \texorpdfstring{$\protect\widetilde{P}$}{\~P} along optimal competitors}
\label{sec:length_estimates}

This section is devoted to collect some estimates involving the functions $P$ and $\wt P$ when they are evaluated on optimal competitors. These estimates play a double role: on the one hand, they lead us to prove that $\wt P$ is positive on optimal competitors; on the other hand, they will be used in the next section to complete the proof of Theorem \ref{thm:main}.

\begin{lemma}
	\label{lem:lb2_larger_1}
    There exists $\e_0>0$ such that, for every $(\e,s)\in\I_{\e_0}$ and every $\w\in\CO$, the first simple loop $\ell$ of $\w$ satisfies 
	\begin{align}
		\label{eq:lb2_larger1} |\lambda_\w|\beta_\ll^2\geq \frac14.
	\end{align}
\end{lemma}

\begin{proof}
    Let $\e_0>0$ be as in Corollary \ref{cor:uniqueness_of_massa_del_sole}, and fix an optimal competitor $\w \in \CO$, with $(\e,s)\in \I_{\e_0}$. Let $\ell$ be the first loop of $\w$, with loop interval $J_\ell$. For every $t\in I_\w$, define $F(t)\ceq  \lambda P(\w(t))^2-\sin(\theta(t))$, where $\theta$ is the angle map of $\w$. Using \eqref{eq:teta} and \eqref{eq:normalpolar}, we compute  
    \begin{align*}
        F'(t) &= 2\lambda P(\w(t)) \left(2\w_1(t)\cos(\theta(t))-b \w_2(t)^{b-1}\sin(\theta(t))\right) - \dot\theta(t)\cos(\theta(t))\\
        &= -\dot\theta(t)\frac{b}2\frac{\w_2(t)^{b-1}}{\w_1(t)}\sin(\theta(t)). 
    \end{align*}
    Hence, for every time interval $[t_1,t_2]\subset J_\ell$, with $t_1<t_2$, we deduce that
    \begin{equation}
        \left|F(t_2)-F(t_1)\right| \leq \int_{t_1}^{t_2} \left| \dot\theta(t)\frac{b}2   \frac{\w_2(t)^{b-1}}{\w_1(t)}\sin(\theta(t))\right| \leq \pi b  \frac{(2\e)^{b-1}}{C\e^q},
    \end{equation}
    where, in the last inequality, we used Proposition \ref{prop:loops>eps/2} and Lemma \ref{lem:optimal}\ref{lem:optimal_item_ii} to estimate $\w_1\geq C\e^q$, for a suitable $C>0$, and we also used Lemma \ref{lem:preliminary}\ref{lem_prelim_item_i} to estimate $\w_2\leq 2\e$ and Lemma \ref{lem:fortissimo}\ref{item:lem_fortissimo_i} to estimate $\int_{t_1}^{t_2}|\dot\theta|dt\leq 2\pi$.  Therefore, for every time interval $[t_1,t_2]\subset J_\ell$, we have
    \begin{equation}
        \left|F(t_2)-F(t_1)\right| \leq C_0 \e^{q-1},
    \end{equation}
    where $C_0=\pi b2^{b-1}C^{-1}$. By definition of $F$, the latter inequality implies that
    \begin{equation}
        \left|\lambda \left(P^2(\w(t_2))-P^2(\w(t_1))\right)\right|\geq |\sin(\theta(t_2))-\sin(\theta(t_1))|-C_0\e^{q-1}.	    
    \end{equation}
	By \eqref{eq:gauss-bonnet-thm}, the times $t_1,t_2\in J_\ell$ can be chosen so that $|\sin(\theta(t_2))-\sin(\theta(t_1))|\geq 1$. Thus, up to shrinking $\e_0$, we get 
	\[
	\Big|\lambda \Big(P^2(\w(t_2))-P^2(\w(t_1))\Big)\Big|\geq 1/2.
	\]
	The conclusion now follows by observing that $|P^2(\w(t_2))-P^2(\w(t_1))|\leq 2\beta_\ell^2$.
\end{proof}

\begin{lemma}
	\label{lem:loc_max_P} 
    There exists $\e_0>0$ such that, for all $(\e,s)\in\I_{\e_0}$ and all $\w\in\CO$, with loop $\ell$, the following holds. Let $t_*\in I_\w\setminus J_\ell$ is a local maximum for $t\mapsto \wt P(\w(t))$ satisfying $\wt P(\w(t_*))>0$. Assume that either $t_*\in I_0$, or $\w_2(\tau_i)\leq\e$, for all $i\in \I$. Then, we have
	\begin{equation}
		\label{eq:locmaxP}
		|\lambda| \wt P(\w(t_*))^{\frac{q+1}{q}}\leq \frac{b(b-1)}{8}.
	\end{equation}
\end{lemma}

\begin{proof}
    By Corollary \ref{cor:unique_intersection_x1axis}, there is a unique $\bar t\in I_0$ such that $\w_2(\bar t)=0$ and $\w_1|_{(0,\bar t]}$ is strictly increasing and positive, and thus the same holds for $\wt P\circ \w|_{(0,\bar t]}$. Therefore, it must be $t_*\geq \bar t$, which implies $\wt P(\w(t_*))=P(\w(t_*))$, and $t_*$ is a local maximum for $P\circ \w$. If we show that, under our assumptions,  
    \begin{equation}
        \label{eq:sin(teta(t_*))>0}
        \sin(\0(t_*))>0,
    \end{equation}
    then, the conclusion follows by \cite[Lem.\ 3.14]{notall}. We now prove \eqref{eq:sin(teta(t_*))>0}. Assume by contradiction that $\sin(\0(t_*))\leq 0$. Since $\w$ is tangent at $t_*$ to the graph of
$\Gamma_\z$, with $\z\ceq P(\w(t))$, cf.\ Section \ref{sec:level_sets}, then $\sin(\theta(t_*))<0$ and $\cos(\0(t_*))<0$ as well. Let $i\in\I_+$ be such that $t_*\in I_i$ and define the curve
\[
\eta:(-\oo,t_*]\to\R^2, \qquad \eta(t)\ceq 
\begin{cases}
    \pr(\gamma)(t), \quad &t\in(-\oo,\tau_i],
    \\
    \w(t), \quad &t\in(\tau_i,t_*],
\end{cases}
\]
and the half-line $r(t)\ceq   \w(t_*)+t\dot\w(t_*)$, for $t\geq0$. Since $\sin(\0(t_*)),\cos(\0(t_*))<0$, it follows that $\spt(\eta)\cap\spt(r)\neq\emptyset$. Then, the minimum $t_1\ceq \min\{t>0\mid r(t)\in \spt(\eta)\}$ is well-defined and is strictly positive since $\w$ is injective at $t_*$ (recall that, by hypothesis, $t_*\notin J_\ell$). Define the curve
\begin{equation}
    \bar \eta:[t_0,t_*+t_1]\to\R^2, \qquad 
    \bar\eta(t)\ceq 
\begin{cases}
    \eta(t), \quad &t\in [t_0,t_*],
    \\
    r(t-t_*), \quad &t\in(t_*,t_*+t_1 ],
\end{cases}
\end{equation}
where $t_0\in (-\oo,t_*)$ is such that $\eta(t_0)=r(t_1)$. By construction, $\bar\eta$ is closed, simple, piecewise smooth, and with non-positive curvature. Moreover, $\w(t_*+\delta)\in \D(\bar\eta)$ for sufficiently small $\delta>0$, so that $t_2\ceq \min\{t\in(t_*,\tau_{i+1}]\mid \w(t)\in \spt(\bar\eta)\}$ is well-defined. Firstly, by Lemma \ref{lem:w1>0}, $\w_1(t_2)>0$, hence $\w(t_2)\notin \spt(\pr(\g)|_{(-\infty,0]})$. Secondly, $\w(t_2)\notin \spt(\w|_{[\tau_i,t_*]})$, because otherwise $t_*$ would belong to $J_\ell$. Thirdly, we claim that $\w(t_2)\notin \spt(r|_{[0,t_1]})$. Indeed, if $\w(t_2)\in \spt(r|_{[0,t_1]})$ then, by the Gauss-Bonnet Theorem \eqref{eq:gauss-bonnet-thm} and Proposition \ref{prop:lambda>0}, $\w|_{[t_*,t_2)}$ must have a loop, which is then the unique loop $\ell$ of $\w$. By the minimality of $t_2$, $\ell$ is such that $\spt(\ell)\subset \D(\bar\eta)$. Now, if $\ell_1(t)\leq\w_1(t_*)$ for all $t\in J_\ell$, then we reach a contradiction using Lemma \ref{lem:cut_and_paste1}\ref{c_a_p_1_item_i}. Thus, there must exist  $\hat t\in J_\ell$ such that $\ell_1(\hat t)\in \D(\bar\eta) \cap\{x\in\R^2\mid x_1> \w_1(t_*)\}$. Since $\cos(\0(t_*)),\sin(\0(t_*))<0$, there exists $\mu>0$ such that $\ell(\hat t)-(0,\mu )\in \spt(\w|_{[\tau_i,t_*)})$, which is again a contradiction by Lemma \ref{lem:cut_and_paste1}\ref{c_a_p_1_item_iii}. 

We are now in position to conclude the proof of \eqref{eq:sin(teta(t_*))>0}. Since $i\in \I_+$, by Corollary \ref{cor:uniqueness_of_massa_del_sole}, either $i=0$ or $i=2$. If $i=0$, the argument above shows that $\w(t_2)\notin \spt(\bar \eta)$ which is a contradiction. If instead $i=2$ (and thus $t_*\in I_2$), then by hypothesis, $\w_2(\tau_2)\leq \e$. If, by contradiction, $\w(t_2)\in \spt(\pr(\g)|_{[0,\w_2(\tau_2))})$, then, $t_2=\tau_3=L(\w)$ (note that by Corollary \ref{cor:unique_intersection_x1axis}\ref{cor:uniqueness_of_massa_del_sole_item_iii}, $\I=\{0,1,2\}$), and thus $\e=\w_2(L(\w))=\w_2(t_2)<\w_2(\tau_2)\leq \e$, which is a contradiction.
\end{proof}

The next proposition shows that $\wt P$ is positive along an optimal competitor. This implies that the only case of Corollary \ref{cor:uniqueness_of_massa_del_sole} that can happen is item \ref{cor:uniqueness_of_massa_del_sole_item_i}, i.e., the one where $J_\ell\subset I_0=I_\w$. 

\begin{proposition}
	\label{prop:P>0}
	There exists $\e_0>0$ such that, for all $(\e,s)\in \mc I_{\e_0}$ and for all $\w\in\CO$, $\wt P(\w(t))>0$, for every $t\in (0,L(\w))$. 
\end{proposition}

\begin{proof}
    For the sake of simplicity, in this proof, $C>0$ denotes a generic constant depending only on $b$, whose value may change from line to line. Let $\e_0>0$ be small enough such that all the previous statements hold. Take $(\e,s)\in\I_{\e_0}$ and $\w\in\CO$. From Lemma \ref{lem:1stloop_is_+_or_-} and Corollary \ref{cor:uniqueness_of_massa_del_sole}, $\w$ admits a unique loop $\ell$, and, either $J_\ell\subset I_0$ (and then $I_0=I_\w$) or $J_\ell\subset I_1$. In particular, $\wt P(\w(t))>0$ for all $t\in(0,L(\w))$ if and only if $J_\ell\subset I_0$. Suppose, by contradiction, $J_\ell\subset I_1$ and we distinguish the two cases where $\w_2(\tau_2)\leq\e$ or $\w_2(\tau_2)>\e$.
    
    \noindent \textsc{Case 1: $\w_2(\tau_2)\leq\e$.} Set $\wt\beta_+\ceq \max_{t\in I_\w} \wt P(\w(t))$ and let $t_*\in [0,L(\w)]\setminus I_1$ be such that $\wt P(\w(t_*))=\wt \beta_+$.  Since $t_*$ is a local maximum of $\wt P\circ \w$ with $\wt P(\w(t_*))>0$ and $t_*\notin J_\ell$, Lemma \ref{lem:loc_max_P} implies that
    \begin{equation}
        |\lambda|\wt\beta_+^{\frac{q+1}{q}}\leq C.
    \end{equation}
    Combining the latter inequality with Lemma \ref{lem:lb2_larger_1}, we conclude that
    \begin{equation}
    \label{eq:beta_+_beta_ell}
        \wt\beta_+\leq C\beta_\ell^\frac{b}{q+1}.
    \end{equation}   
    We apply Proposition \ref{prop:sublevel_length_estimate} to the curve $\eta\ceq \w|_{[0,s_\ell^-]}\ast\w|_{[s_\ell^+,L(\w)]}$, with $\varrho\ceq \wt \beta_+$, to deduce that 
    \begin{equation}
        L(\w)= L(\eta)+L(\ell) \geq 
        L(\gamma_{s,\e}) - C \wt\beta_+^{1-\frac{1}{b}} +L(\ell) \geq 
        L(\w) - C \beta_\ell^{\frac{b-1}{q+1}} +L(\ell),   
    \end{equation}
    where, in the last inequality, we used \eqref{eq:beta_+_beta_ell} and the fact $\w \in\CO$. This, in turn, implies that 
    \begin{equation}
        L(\ell) \leq  C \beta_\ell^{\frac{b-1}{q+1}}\leq C \beta_\ell, 
    \end{equation}
    where in the last inequality we used that $b\geq4$. Conversely, we can find a lower bound on $L(\ell)$, combining Proposition \ref{prop:L(1stloop)>dist2}, together with Lemma \ref{lem:preliminary}\ref{lem_prelim_item_i}-\ref{lem_prelim_item_i_bis}, obtaining
    \begin{equation}
        L(\ell) \geq C \beta_\ell\e^{-q}.
    \end{equation}
    The last two inequalities lead to the contradiction $\e^{-q}\leq C$.

    \noindent \textsc{Case 2: $\w_2(\tau_2)>\e$.} In this case, we set $\wt\beta_+\ceq \max_{t\in [0,\tau_2]} \wt P(\w(t))$. Let $t_*\in [0,\tau_2]\setminus I_1=I_0$ be such that $\wt P(\w(t_*))=\wt \beta_+$. As before, $t_*\in I_0\setminus J_\ell$ is a local maximum of $\wt P\circ \w$ with $\wt P\circ \w(t_*)>0$. Then, reasoning as in the previous case, Lemmas \ref{lem:lb2_larger_1} and \ref{lem:loc_max_P} imply 
    \begin{equation}
        \label{eq:75primo}
        \wt\beta_+\leq C\beta_\ell^\frac{b}{q+1}.
    \end{equation}
    We apply Proposition \ref{prop:sublevel_length_estimate} to the curve $\eta\ceq \w|_{[0,s_\ell^-]}\ast\w|_{[s_\ell^+,\tau_2]}$, to deduce that 
    \begin{equation}
        L(\eta)\geq L(\gamma_{s,\w_2(\tau_2)}) - C \wt\beta_+'^{1-\frac{1}{b}} \geq L(\w) - C \beta_\ell^{\frac{b-1}{q+1}}, 
    \end{equation}
    where, in the last inequality, we used \eqref{eq:75primo} and the estimate $L(\gamma_{s,\w_2(\tau_2)})\geq L(\gse)\geq L(\w)$ which holds since $\w(\tau_2)>\e$ and $\w\in\CO$. Therefore, we get
    \begin{equation}
        L(\w)= L(\eta)+L(\ell) +L(\w|_{I_2})
        \geq
        L(\eta) + L(\ell)
        \geq
        L(\w)-C \beta_\ell^{\frac{b-1}{q+1}} + L(\ell).  
    \end{equation}
    We reach a contradiction reasoning as in the previous case.
\end{proof}

\section{Proof of the main theorem}
\label{sec:proof}

Along this section $C>0$ denotes a 
generic constant depending only on $b$, and its value may change from line to line. Let $\e_0>0$ be such that all the statements in the previous sections hold, and fix $(\e,s)\in\I_{\e_0}$, and $\w\in\CO$. Recall that $\w$ has a unique loop $\ell$ and, by Proposition \ref{prop:P>0}, $J_\ell\subset I_0=[0,L(\w)]$, and thus, $\wt P\circ\w\geq0$ and $\wt\beta=\max_{t\in I_\w}\wt P(\w(t))$. Define the two quantities
\begin{equation}
\label{eq:zeta} 
\zeta \ceq  \sqrt{\wt\bb  \e^{-b}} \qquad \textrm{and}\qquad 
\z \ceq  \left(\wt\bb \e^\1\right)^{\frac{b}{b-1}}. 
\end{equation}
Note that, since $\wt\beta\e^{-b}<\e^{M-b}<1$ by Corollary \ref{cor:newLS} (recall that $M>b$ fixed), $\varrho<\wt\beta$. In addition, if $t_*\in I_0\setminus J_\ell$ is a local maximum for $P\circ\w$, then combining Lemmas \ref{lem:lb2_larger_1} and \ref{lem:loc_max_P}, we obtain that
\begin{equation}
    \label{eq:locmaxP_final}
    \wt P(\w(t_*))\leq C\wt \beta^{\frac{b}{q+1}} \leq \z <\wt\beta.
\end{equation}
This fact implies, in particular,
\begin{equation}
    \label{eq:beta=betaloop}
    \wt\beta=\beta_\ell.
\end{equation}
At this stage, let us define the curve
\begin{equation}
    \nu\ceq \spt(\w|_{[0,s_\ell^-]})\cup\spt(\w|_{[s_\ell^+,L(\w)]}),
\end{equation}
whose length, by optimality of $\w$, satisfies
\begin{equation}
    \label{eq:start}
    L(\g_{s,\e})\geq L(\w)=L(\ell)+L(\nu).
\end{equation}
We divide the proof into the two cases where $\wt P(\w(s_\ell^\pm))\leq\z$ and $\wt P(\w(s_\ell^\pm))>\z$.

\noindent\textsc{Case 1: $\wt P(\w(s_\ell^-))\leq \z$.} In this case, since $\varrho<\wt\beta$, we may apply Proposition \ref{prop:sublevel_length_estimate} to $\nu$ and estimate $L(\nu)\geq L(\g_{s,\e})-C\z^{1-\frac1b}$. Therefore, by \eqref{eq:start} we get
\begin{equation}
    L(\ell)\leq C\z^{1-\frac1b}= C\wt\beta\e^\1.
\end{equation}
However, by Propositions \ref{prop:L(1stloop)>dist2} and \ref{prop:loops>eps/2}, $L(\ell)\geq C\wt\beta\e^{-q}$, which is a contradiction for small $\e>0$.

\noindent\textsc{Case 2: $\wt P(\w(s_\ell^-))> \z$.} Define $t^-, t^+\in I_\w$ by
\begin{equation}
\label{eq:t-andt+}
    t^-\ceq  \max \Bigl\{t\in [0,s_\ell^-] \, \vert \, \wt P(\omega(t)) =\z \Bigr\} \quad \mbox{and} \quad  t^+\ceq  \min \Bigl\{t\in [s_\ell^+,L(\omega)] \, \vert \, \wt P(\omega(t)) =\z \Bigr\}.
\end{equation}
Then, by construction, $0<t^-<s_\ell^-<t_\ell<s_\ell^+<t^+<L(\w)$. Moreover, by \eqref{eq:locmaxP_final}, we deduce that
\begin{equation}
    \label{eq:liso_loops_trick} 
    \wt P(\w(t))< \z \qquad \mr{for \; all \;} t\in I_0 \setminus(t^-, t^+).
\end{equation}

\begin{lemma}
\label{lem:liso_loops_trick} 
Under the assumptions and notations above, we have $\w_2(t^-)<\w_2 (t^+)$ and
\begin{equation}
    \label{eq:liso_loops_trick_times} t_\ell\left(1-\zeta\right)<t^- \qquad \text{and} \qquad t^+<t_\ell\left(1+\zeta\right).
\end{equation}
\end{lemma}

\begin{proof}
    Combining Corollary \ref{cor:rough_bound_loops} with \eqref{eq:beta=betaloop} and the estimate $t_\ll\geq y_\ll$, we get
\begin{equation}
	\label{ST1}
	\frac{t_\ell -s_\ell^-}{t_\ll\zeta}
	\leq \frac{L(\ell)}{t_\ll\zeta}\leq C \wt\beta  
	^{\frac 12-\frac 1 b} y_\ell^{\frac b 2 -1}<\frac12,
\end{equation}
because $\wt\bb, y_\ll\to0$ as 
$\e\to0$. This shows that $s_\ell^->t_\ell\left(1-\frac\zeta2\right)$. Assume now by contradiction that $t^-<t_\ell(1-\zeta)$ and set $I\ceq [t_\ell(1-\zeta),s_\ell^-]$. Then, by \eqref{eq:t-andt+} and \eqref{eq:liso_loops_trick}, it holds
\begin{equation}
\label{eq:lower_bound_aux_Pw}
	P(\w(t))>\z, \quad \text{for all } t\in I.
\end{equation}
The contradiction is reached by estimating from above and from below the total curvature of $\w$ on $I$. By the Gauss-Bonnet Theorem \eqref{eq:gauss-bonnet-thm}, we can estimate from above
\begin{equation}
\label{eq:total_curvature_upper_bd}
	\int_I |\dot\0(t)| dt  \leq C.
\end{equation}
On the other hand, exploiting \eqref{eq:lower_bound_aux_Pw} and \eqref{ST1}, we can estimate from below
\begin{equation}
	\label{mancante}
	\int_I |\dot\0(t)| dt = 4|\la|\int_I  P(\w(t))\w_1(t) dt \geq 4|\la| |I|\z \min_{t\in I} \w_1(t) \geq 2|\la| t_\ell\,\zeta\,\z \min_{t\in I} \w_1(t). 
\end{equation}
We claim that
\begin{equation}
    \label{eq:stima_min_w1}
    \min_{t\in I} \w_1(t) \geq C\e^q.
\end{equation}
Indeed, by Proposition \ref{prop:loops>eps/2}, we have $\w_2(s_\ell^-)\geq C\e$, and, recalling that $s_\ell^-<t_\ell<\e$, it holds $|I|\leq t_\ell\zeta\leq \e^{1+\frac{M-b}{2}}$, having used Corollary \ref{cor:newLS} and  \eqref{eq:zeta}. Since $M>b$ and $|\dot\w_2|\leq1$, we estimate
\begin{equation}
    \w_2\big(t_\ell(1-\zeta)\big)= \w_2(s_\ell^-)- \int^{s_\ell^-}_{t_\ell(1-\zeta)}\dot\w_2(\tau)d\tau\geq C\e-\e^{1+\frac{M-b}{2}}\geq C\e,
\end{equation}
provided that $\e>0$ is small. The latter inequality, together with Lemma \ref{lem:optimal}(ii), implies \eqref{eq:stima_min_w1}. By \eqref{eq:stima_min_w1} and Lemma \ref{lem:lb2_larger_1}, and recalling \eqref{eq:zeta} and \eqref{eq:beta=betaloop}, we can continue from \eqref{mancante} to deduce that 
\begin{equation}
\label{eq:total_curvature_lower_bd}
	\int_I |\dot\0(t)| dt \geq C|\la| t_\ell\,\zeta\,\z\, \e^q  \geq C\wt\beta^{-2} t_\ell\,\zeta\,\z \,\e^q = C t_\ell \wt\beta^{\frac{b}{b-1}-\frac32}  \e^{-\frac{b}{b-1}}   \geq C \wt\beta^{\frac{b}{b-1}-\frac32}  \e^{-\frac{1}{b-1}},
\end{equation}
where, in the last inequality, we used that $t_\ell\geq y_\ell\geq C\e$, by Proposition \ref{prop:loops>eps/2}. Finally, combining the estimates \eqref{eq:total_curvature_upper_bd} and \eqref{eq:total_curvature_lower_bd}, we reach a contradiction for small $\e>0$, since, for all $b\geq 5$,
\begin{equation}
    \frac{b}{b-1}-\frac32<0 \qquad\text{and}\qquad \wt\beta^{\frac{b}{b-1}-\frac32}  \e^{-\frac{1}{b-1}}\leq C. 
\end{equation} 
We finally prove the inequality $\w_2 (t^-) < \w_2( t^+)$. Assume by contradiction that $\w_2 (t^-)\geq \w_2( t^+)$.
If  $\w_2 (t^-)= \w_2( t^+)$ then it holds that $\w (t^-)= \w( t^+)$, contradicting the uniqueness of the loop. If instead $\w_2 (t^-)> \w_2( t^+)$, let $U$ be the open and bounded region with boundary 
\begin{equation*}
	\Gamma \ceq  \spt(\w|_{[0, t^-]})
    \cup 
    \{x\in\R^2\mid \wt P(x)=\z, x_1>0, -s \leq x_2 \leq \w_2( t^-)\}\cup
    \{(x_1,-s)\in\R^2 \mid 0\leq x_1 \leq \sqrt\z\}.
\end{equation*}
Then, by \eqref{eq:t-andt+} and \eqref{eq:liso_loops_trick}, $\w( t^+ +\delta)\in U$ for $\delta>0$ sufficiently small. Since $\w(L(\w))=\pi(\g(\e))\notin U$, there exists $ \bar t\in ( t^+,L(\w))$ such that $\w(\bar t)\in \dd U$. 
If $\w(\bar t)\in \spt(\w|_{[0,t^-]})$, we contradict the uniqueness of the loop; if $\w(\bar t)\in \{x\in\R^2\mid \wt P(x)=\z, \, x_1>0, \, -s \leq x_2 \leq \w_2( t^-)\}$, we contradict \eqref{eq:liso_loops_trick}; and if $\w(\bar t)\in \{x\in\R^2 \mid x_2=-s, \, 0\leq x_1 \leq \z^{\frac1a}\}$, then $\w_2(\bar t)=-s$ and we reach a contradiction combining Proposition \ref{prop:loops>eps/2} and Lemma \ref{lem:optimal}\ref{lem:optimal_item_i}. 
\end{proof}

Define $s^\pm\ceq \w_2(t^\pm)$, and let $\bar \nu:[0,L(\bar\nu)]\to\R^2$ be the curve given by the concatenation 
\begin{equation}
\bar\nu = \w|_{[0,t^-]}* \G_\z|_{[s^-,s^+]} * \w|_{[t^+,L(\nu)]},
\end{equation}
where $\Gamma_\z(t) =(f_\z(t),t)=\big(\sqrt{t^b+\z},t\big)$, cf.\ Section \ref{sec:level_sets}.
By \eqref{eq:liso_loops_trick}, it holds $\spt(\bar\nu) \subset E_\z$ and thus, applying Proposition~\ref{prop:sublevel_length_estimate}, we estimate 
\begin{equation} \label{D2}
L(\bar\nu)\geq L(\gamma_{s,\e}) -  C \z^{1-\frac 1 b} \geq L(\ell)+L(\nu)-  C \z^{1-\frac 1 b},
\end{equation} 
where we used \eqref{eq:start} in the second inequality. Hence, we obtain the following lower bound on $\varrho$: 
\begin{equation}
    \label{D201}
    C\z^{1-\frac1b}\geq L(\ell) + L(\nu) - L(\bar\nu)\geq L(\ell) + L([\Gamma_\z(s^-),\Gamma_\z (s^+)]) - L(\Gamma_\z|_{[s^-,s^+]}),
\end{equation}
where $[\Gamma_\z(s^-),\Gamma_\z (s^+)]$ is the segment joining $\Gamma_\z(s^-)$ to $\Gamma_\z (s^+)$. To estimate from below the latter quantity, we use \cite[Prop.\ 2.1(ii)]{notall}, whose proof is unchanged in our setting: for all $0<\bar t<t$, and $\z>0$, we have
\begin{equation}
	\label{eq:ls-segm}
	L\left(\G_\z|_{[\bar t,t]}\right)-L([\G_\z(\bar t),\G_\z(t)])\leq \frac 12 q^2(q-1) t^{b-1}\left(1-\frac{\bar t}{t}\right)^ 2 .
\end{equation}
In particular, we apply \eqref{eq:ls-segm} with $t=s^+$ and $\bar t=s^-$, obtaining
\begin{equation}\label{D4}
	L([\Gamma_\z(s^-),\Gamma_\z (s^+)]) - L(\Gamma_\z|_{[s^-,s^+]})  
	\geq - C  (s^+)^{b-1}\Big(1-\frac{s^-}{s^+}   \Big) ^2.
\end{equation}
As $t_+>t_\ell$, by Lemma \ref{lem:optimal}\ref{lem:optimal_item_i} and Proposition \ref{prop:loops>eps/2}, it holds $s^+\geq C\e$. Moreover, by Lemma \ref{lem:liso_loops_trick}, we deduce that $0\leq s^+-s^- =\omega_2( t^+)-\omega_2( t^-)\leq     
t^+- t^-\leq 2  t_\ell \zeta\leq 4 \e\zeta$. Thus, we obtain 
\begin{equation}
\label{eq:boundxi}
0\leq 1-\frac{s^-}{s^+}\leq C\zeta.
\end{equation}
From \eqref{D4} and \eqref{eq:boundxi}, and from the definition of $\zeta$ in \eqref{eq:zeta}, we finally deduce 
\begin{equation}
\label{eq:bound_ls-segm}
L([\Gamma_\z(s^-),\Gamma_\z (s^+)]) - L(\Gamma_\z|_{[s^-,s^+]}) \geq -C\e^{b-1}\zeta^2=-C\e^{-1}\wt\bb.
\end{equation}
We are in position to complete the proof of Theorem \ref{thm:main}. From \eqref{D2} and \eqref{eq:bound_ls-segm}, we get  
\begin{equation}
L(\ell)\leq  C\left(\z^{1-\frac 1b} +\e^{-1} \wt\bb \right)=
2C\e^{-1}\wt\bb.
\end{equation}
From the latter estimate and Proposition \ref{prop:L(1stloop)>dist2}, we derive the inequality
\begin{equation}
\label{D5}
\e^{-q} \wt\bb  \leq  C\e^{-1}\wt\bb,
\end{equation}
which is a contradiction for small $\e>0$.

\section{Proof of Theorem \ref{thm:lift}}
\label{sec:lift}

Recall that $\mc M$ is the \sr manifold $(\R^3,\De,g)$, where $\De=\vspan\{X_1,X_2\}$ is defined in \eqref{eq:horvf} and $g$ is the metric obtained by declaring $\{X_1,X_2\}$ orthonormal. Fix $s,\varepsilon>0$ such that the curves $\gse$ and $\bar\g_{s,\e}$ are minimizing geodesics in $\mc M$, see Theorem \ref{thm:main}. The proof of Theorem \ref{thm:lift} consists in the construction of an explicit Carnot group $G$ (see \cite[Ch.\ 11]{libro_Enrico} for an introduction to Carnot groups) and a submetry $\phi:G\to\mc M$ (see \cite[Def. 3.1.23]{libro_Enrico}). It follows from~\cite[Sec.~3.1.7.1]{libro_Enrico} that $\gse$ and $\bar\g_{s,\e}$ admit a lift to some geodesic in $G$. As the lift preserves the regularity, this shows the existence of non-smooth geodesics in Carnot groups. Then, we show that the lifted curves coincide on $[-s,0]$, proving that branching of geodesics may occur in Carnot groups as well.  

In $\R^3$ with coordinates $z=(z_1,z_2, z_3)$, we consider the vector fields
\begin{equation}
    \begin{split}
        &Z_1(z)\ceq\partial_1\\
        &Z_2(z)\ceq\partial_2 + z_1(z_1^3 - 2z_1z_2^b)\partial_3,
    \end{split}
\end{equation}
and we define the sub-Riemannian manifold $\mc N\ceq(\R^3,\bar\De,\bar g)$, where $\bar\De\ceq\vspan\{Z_1,Z_2\}$ and $\bar g$ is the metric obtained by declaring $\{Z_1,Z_2\}$ orthonormal. The smooth diffeomorphism
\begin{equation}
    F:\R^3\to\R^3, \quad F(z)=\Big(z_1,z_2, z_3+\frac{z_2^{2b+1}}{2b+1}\Big),
\end{equation}
satisfies $F_*Z_i=X_i$, for $i=1,2$, and thus it is an isometry between $\mc M$ and $\mc N$. In particular, the curves $\G\ceq F^{-1}(\gse)$ and $\bar\G\ceq F^{-1}(\bar\g_{s,\e})$ are minimizing geodesics in $\mc N$ and they satisfy 
\begin{equation}
    \G(t)=\bar\G(t)=(0,t,0), \quad \text{for all } t\in[-s,0].
\end{equation} 
Therefore, it is enough to construct a submetry from a suitable Carnot group $G$ to $\mc N$ such that the lifts of $\G,\bar\G$ coincide on $[-s,0]$.

Consider the set of multi-indices $\Lambda_b\ceq \{\a=(\alpha_1,\alpha_2)\in\N^2\mid \alpha_1+\alpha_2 \leq b+2\}$ of cardinality $d= \frac{(b+4)(b+3)}{2}$. In $\R^{d+2}$, fix the coordinates $y=\big(y_1,y_2,(y_\a)_{\a\in\Lambda_b}\big)$, where $\Lambda_b$ is equipped with the lexicographic order. Let $G=(\R^{d+2},\cdot)$ be the Carnot group equipped with the left-invariant orthonormal frame which, in exponential coordinates of the second type, takes the form:
\begin{equation}
    \begin{split}
        &Y_1(y)\ceq\partial_1,\\
        &Y_2(y)\ceq\partial_2 + y_1\sum_{\alpha\in \Lambda_b} c_{\a} y_1^{\a_1} y_2^{\a_2} \ \partial_{\alpha},
    \end{split}
\end{equation}
where $\partial_\alpha := \partial_{y_{\alpha}}$, and where $c_{\a}=1$ for every $\a\neq (1,b)$ and $c_{(1,b)}=-2$.
% Denote by $\Upsilon$ the set of multi-indices $\Upsilon\ceq \{\mu=(i_1,\ldots,i_k,1,2)\mid k\geq0 \ i_1,...,i_k=1,2\}$ and, for $\mu\in\Upsilon$, let us denote $Y_{\mu}\ceq[Y_{i_1},[Y_{i_2},...,[Y_{i_{k-1}},Y_{i_k}]]]$. In particular if $\mu\in\Upsilon$ counts $k_1+1$ ones and $k_2+1$ twos, with $k_1,k_2\geq0$ and $k_1+k_2=k$, then we have $Y_\mu(0)=(k_1+1)!\, k_2!\, \partial_{(k_1,k_2)}$. It follows that the Lie algebra of $G$, denoted by $\mf g$, is generated (as a linear vector space) by
% \begin{equation}
%     \vspan\{\partial_1,\partial_2,(\partial_\a)_{\a\in\Lambda_b}\}=\vspan\{Y_1(0),Y_2(0),(Y_\mu(0))_{\mu\in\Upsilon}\}.
% \end{equation}
% In particular, $\{\partial_1,\partial_2,(\partial_\a)_{\a\in\Lambda_b}\}$ is a linear basis for $\mf g$.
Let $\mf g$ be the Lie algebra of $G$, and note that $\{\partial_1,\partial_2,(\partial_\a)_{\a\in\Lambda_b}\}$ is a basis for $\mf g$. Define $\Lambda_b^0\ceq \{(3,0),(1,b)\}\subset \Lambda_b$ and  
\begin{equation}
    \mf h:=\vspan\{(\partial_\a)_{\a\in\Lambda_b\setminus \Lambda_b^0}, \partial_{(3,0)}-\partial_{(1,b)}\}\subset \mf g.
\end{equation} 
Since $[\mf g,\mf g]$ is abelian, we have that $\mf h$ is a sub-algebra of $\mf g$ contained in $[\mf g, \mf g]$, and thus the set $H\ceq \exp(\mf h)\subset G$ is a closed subgroup of $G$. We note that every element $h\in H$, in the $y$-coordinates, has the form (recall that $\Lambda_b$ has the lexicographic order)
\begin{equation}
    h=\big( 0,0,h_{(0,0)}, \ldots , h_{(1,b-1)}, t , h_{(1,b+1)}, \ldots, h_{(2,b)}, -t, h_{(3,1)}, \ldots, h_{(b+2,0)} \big), \quad h_\a,t\in\R, \  \a\in\Lambda_b\setminus\Lambda_b^0,
\end{equation}
and an element $yH\in G\slash H$, $y=(y_1,y_2,(y_\a)_{\a\in\Lambda_b})\in G$, is represented by
\begin{equation}
    yH = \Big[\Big(y_1,y_2, 0, \ldots , 0, \frac12(y_{(1,b)}+y_{(3,0)}), 0, \ldots, 0, \frac12(y_{(1,b)}+y_{(3,0)}), 0, \ldots 0 \Big)\Big]. 
\end{equation}
Thus, we have the identification $G\slash H \ni yH \mapsto (y_1,y_2, y_{(3,0)}+y_{(1,b)})\in \R^3$ of the quotient manifold $G\slash H$ with $\R^3$, and the projection $\bar\pi: G\to \R^3$ defined by 
\begin{equation}
    \bar\pi(y)=(y_1,y_2, y_{(3,0)}+y_{(1,b)}).
\end{equation}
A direct computation shows that, for every $z\in\R^3$ and every $y\in \pi^{-1}(z)$, we have
\begin{equation}
    \label{eq:dpi_costante_fibre}
    \begin{split}
        &\df_y\bar\pi(Y_1(y))=Z_1(z),
        \\
        &\df_y\bar\pi(Y_2(y))=Z_2(z).
    \end{split}
\end{equation}
By \cite[Prop. 7.1.9]{libro_Enrico}, we conclude that the map $\bar \pi$ is a submetry from $G$ to $\mc N$. Fix $p=(0,-s,\ast)\in\bar\pi^{-1}(\G(-s))=\bar\pi^{-1}(\bar\G(-s))\subset G$. By \cite[Cor. 3.1.25]{libro_Enrico}, there are two geodesics $\Psi,\bar\Psi:[-s,\e]\to G$ such that $\Psi(-s)=\bar\Psi(-s)=p$, $\bar\pi(\Psi)=\G$ and $\bar\pi(\bar\Psi)=\bar\G$. The regularity of the curves $\Psi$ and $\bar\Psi$ is the same of their projections. We claim that they also branch, namely
\begin{equation}
    \Psi(t)=\bar\Psi(t), \quad \text{for all } t\in [-s,0].
\end{equation}
Indeed, let $\psi$ be the control of $\Psi$ and let $v=(v_1,v_2)\equiv(0,1)$ be the control of $\G$, on $[-s,0]$. By differentiating the identity $\bar\pi(\Psi)=\G$ on $[-s,0]$ and using \eqref{eq:dpi_costante_fibre}, we get
\begin{equation}
    \partial_2=\dot\G(t)=\frac{\df}{\df t}\bar\pi(\Psi(t))=\df_{\Psi(t)}\bar\pi(\dot\Psi(t)) = \psi_1(t)Z_1(\G(t)) + \psi_2(t)Z_2(\G(t)) = \psi_1\partial_1+\psi_2\partial_2.
\end{equation}
This gives $\psi=(\psi_1,\psi_2)\equiv(0,1)$ on $[-s,0]$. Similarly, we have $\bar\psi=(\bar\psi_1,\bar\psi_2)\equiv(0,1)$ on $[-s,0]$, and thus the geodesics $\Psi$ and $\bar\Psi$ coincide on $[-s,0]$. This completes the proof of Theorem \ref{thm:lift}.

\bibliographystyle{abbrv}
\bibliography{biblio}

\end{document}

%% file: preambolo_gg.tex
\usepackage[utf8]{inputenc}
\usepackage[english]{babel}

\usepackage{mathrsfs}
\usepackage{amsthm}
\usepackage{amsmath}
\usepackage{amsfonts}
\usepackage{mathtools}
\usepackage{amssymb}

\usepackage{textcomp}
\usepackage{xcolor}

\usepackage{calrsfs}
\usepackage{csquotes}

\usepackage{tikz-cd}

\usepackage{comment}
\usepackage[shortlabels]{enumitem}

\DeclareMathAlphabet{\mathcal}{OMS}{cmsy}{m}{n}

\theoremstyle{plain}
\newtheorem {theorem}{Theorem}[section]
\newtheorem {lemma}[theorem]{Lemma}
\newtheorem {corollary} [theorem]{Corollary}

\newtheorem {proposition} [theorem]{Proposition}
\theoremstyle{definition}
\newtheorem{definition}[theorem]{Definition}
\newtheorem{remark}[theorem]{Remark}

 \numberwithin{equation}{section} 
 
\newcommand{\Z}{\mathbb{Z}}

\newcommand{\G}{\Gamma}
\newcommand{\R}{\mathbb{R}}

\newcommand{\N}{\mathbb{N}}

\newcommand{\D}{\mathcal{D}}

\newcommand{\I}{\mathcal{I}}

\newcommand{\E}{\mathcal E}

\newcommand{\oo}{\infty}

\newcommand{\e}{\varepsilon}
\newcommand{\s}{\sigma}
\newcommand{\0}{\theta}
\newcommand{\w}{\omega}

\newcommand{\dd}{\partial}
\renewcommand{\a}{\alpha}
\newcommand{\bb}{\beta}
\renewcommand{\k}{\kappa}
\newcommand{\g}{\gamma}
\newcommand{\la}{\lambda}
\newcommand{\de}{\delta}
\newcommand{\De}{\Delta}

\newcommand{\z}{\zeta}

\newcommand{\1}{{-1}}

\newcommand{\sgn}{\text{sgn}}
\renewcommand{\ll}{\ell}

\newcommand{\vspan}{\mathrm{span}}

\newcommand{\mc}{\mathcal}
\newcommand{\mr}{\mathrm}

\newcommand{\wt}{\widetilde}

\newcommand{\be}{\begin{equation}}
\newcommand{\ee}{\end{equation}}

\renewcommand{\z}{{\varrho}}
\renewcommand{\k}{\mathrm k}

\newcommand{\spt}{{\mr{spt}}}